\documentclass[11pt]{amsart}
\usepackage{amsmath}
\usepackage{amssymb}
\usepackage{amscd}
\usepackage{enumerate}
\usepackage{verbatim}
\usepackage{color}
\usepackage{bm}
\usepackage{easybmat} 
\usepackage{fullpage}

\newtheorem{Thm}{Theorem}[section]
\newtheorem{Lemma}[Thm]{Lemma}
\newtheorem{Cor}[Thm]{Corollary}
\newtheorem{Prop}[Thm]{Proposition}
\theoremstyle{definition}
\newtheorem{Def}[Thm]{Definition}
\newtheorem{Example}[Thm]{Example}

\newtheorem{Rmk}[Thm]{Remark}

\numberwithin{equation}{section}

\def\bbz{\mathbb{Z}}
\def\bbc{\mathbb{C}}
\def\bbr{\mathbb{R}}
\def\bbz{\mathbb{Z}}
\def\bbn{\mathbb{N}}
\def\bbq{\mathbb{Q}}

\def\bbf{\mathbb{F}}

\def\lra{\longrightarrow}

\def\x{\times}
\def\bs{\backslash}

\def\aut{\mathrm{Aut}}
\def\aff{\mathrm{Aff}}

\def\fix{\mathrm{Fix}}

\def\id{\mathrm{id}}
\def\tr{\mathrm{tr\,}}

\def\R{(\mathrm{R})}
\def\Ad{\mathrm{Ad}}
\def\ad{\mathrm{ad}}
\def\sp{\mathrm{sp}}

\def\frakS{\mathfrak{S}}

\def\Sol{\mathrm{Sol}}

\def\sgn{\mathrm{sign~\!}}

\def\R{\mathrm{(R)}}

\def\frakG{\mathfrak{G}}

\def\calS{\mathcal{S}}

\def\HPer{\mathrm{HPer}}
\def\DH{\mathrm{DH}}

\def\lcm{\mathrm{lcm}}

\def\calA{\mathcal{A}}

\def\bb{|\hspace{-1pt}|}

\def\calO{\mathcal{O}}
\def\Per{\mathrm{Per}}

\def\EP{\mathrm{E}P}
\def\DA{\mathrm{DA}}
\def\DP{\mathrm{DP}}
\def\ds{\displaystyle}

\def\boxit#1{\vbox{\hrule\hbox{\vrule\kern3pt
     \vbox{\kern3pt#1\kern3pt}\kern3pt\vrule}\hrule}}

\begin{document}
\title[The Nielsen numbers of iterations of maps on infra-solvmanifolds]
{The Nielsen numbers of iterations of maps on infra-solvmanifolds of type $\R$ and periodic orbits}
\author{Alexander Fel'shtyn}
\address{Instytut Matematyki, Uniwersytet Szczecinski,
ul. Wielkopolska 15, 70-451 Szczecin, Poland}
\email{fels@wmf.univ.szczecin.pl}

\author{Jong Bum Lee}
\address{Department of mathematics, Sogang University, Seoul 121-742, KOREA}
\email{jlee@sogang.ac.kr}

\thanks{The first-named  author is indebted to the Max-Planck-Institute for Mathematics(Bonn) for the support and hospitality.
The second-named author is
supported in part by Basic Science Researcher Program through the National Research Foundation of Korea(NRF) funded by the Ministry of Education (No.~\!2013R1A1A2058693)}


\abstract
We study the asymptotic behavior of the sequence of the Nielsen numbers $\{N(f^k)\}$, the {essential periodic orbits} of $f$ and the homotopy minimal periods of $f$ by using the Nielsen theory of maps $f$ on infra-solvmanifolds of type $\R$. We give a linear lower bound for the number of essential periodic orbits of such a map, which sharpens well-known results of Shub and Sullivan for periodic points and of Babenko and Bogaty\v{\i} for periodic orbits. We also verify that a constant multiple of infinitely many prime numbers occur as homotopy minimal periods of such a map.
\endabstract
\date{\today}

\maketitle

\tableofcontents

\section{Introduction}

Let $f:X\to X$ be a map on a finite complex $X$. A point $x\in X$ is a fixed point of $f$ if $f(x)=x$; a periodic point of $f$ with period $n$ if $f^n(x)=x$. The smallest period of $x$ is called the minimal period. We will use the following notations:
\begin{align*}
&\fix(f)=\{x\in X\mid f(x)=x\},\\
&\Per(f)=\text{the set of all minimal periods of $f$},\\
&P_n(f)=\text{the set of all periodic points of $f$ with minimal period $n$},\\
&{\HPer(f)=\bigcap_{g\simeq f}\left\{n\in\bbn\mid P_n(g)\ne\emptyset\right\}}\\
&\hspace{1.4cm}=\text{the set of all homotopy minimal periods of $f$}.
\end{align*}
Let $p:\tilde{X}\rightarrow X$ be the universal cover of $X$ and $\tilde{f}:\tilde{X}\rightarrow \tilde{X}$ a lift of $f$, {i.e.,} $p\circ\tilde{f}=f\circ p$. Two lifts $\tilde{f}$ and $\tilde{f}^\prime$ are called \emph{conjugate} if there is a $\gamma\in\Gamma\cong\pi_1(X)$ such that $\tilde{f}^\prime = \gamma\circ\tilde{f}\circ\gamma^{-1}$. The subset $p(\fix(\tilde{f}))\subset \fix(f)$ is called the \emph{fixed point class} of $f$ determined by the lifting class $[\tilde{f}]$. A fixed point class is called \emph{essential} if its index is nonzero. The number of essential fixed point classes is called the \emph{Nielsen number} of $f$, denoted by $N(f)$ \cite{Jiang}.

The Nielsen number is always finite and is a homotopy invariant lower bound for the number of fixed points of $f$. In the category of compact, connected polyhedra the Nielsen number of a map is, apart from in certain exceptional cases, equal to the least number of fixed points of maps with the same homotopy type as $f$.

From the dynamical point of view, it is natural to consider the Nielsen numbers $N(f^k)$ of all iterations of $f$ simultaneously. For example, N. Ivanov \cite{I} introduced the notion of the asymptotic Nielsen number, measuring the growth of the sequence $N(f^k)$ and found the basic relation between the topological entropy of $f$ and the asymptotic Nielsen number. Later on, it was suggested in \cite{Fel84, PilFel85, Fel88, Fel00} to arrange the Nielsen numbers $N(f^k)$ of all iterations of $f$ into the Nielsen zeta function
$$
N_f(z)=\exp\left(\sum_{k=1}^\infty\frac{N(f^k)}{k}z^k\right).
$$
The Nielsen zeta function $N_f(z)$ is a nonabelian analogue of the Lefschetz zeta function
$$
L_f(z)=\exp\left(\sum_{k=1}^\infty\frac{L(f^k)}{k}z^k\right),
$$
where
\begin{equation*}
 L(f^n) := \sum_{k=0}^{\dim X} (-1)^k\ \tr\!\left\{f_{*k}^n:H_k(X;\bbq)\to H_k(X;\bbq)\right\}
\end{equation*}
is the Lefschetz number of the iterate $f^n$ of $f$.

Nice analytical properties of $N_f(z)$ \cite{Fel00} indicate that the numbers $N(f^k)$ are closely interconnected.
The manifestation of this are Gauss congruences
\begin{align*}
\sum_{d\mid k}\mu\!\left(\frac{k}{d}\right)N(f^d)\equiv0\mod{k},
\end{align*}
for any $ k>0$,  where $f$ is  a map on an infra-solvmanifold of type $\R$ \cite{FL}.

The fundamental invariants of $f$ used in the study of periodic points are the Lefschetz numbers $L(f^k)$, and their algebraic combinations, the Nielsen numbers $N(f^k)$ and the Nielsen-Jiang periodic numbers $NP_n(f)$ and $N\Phi_n(f)$.

The study of periodic points by using the Lefschetz theory has been done extensively by many authors in the literatures such as \cite{Jiang}, \cite{Dold}, \cite{BaBo}, \cite{JM}, \cite{Matsuoka}. A natural question is to know how much information we can get about the set of essential periodic points of $f$ or about the set of (homotopy) minimal periods of $f$ from the study of the sequence $\{N(f^k)\}$ of the Nielsen numbers of iterations of $f$. Even though the Lefschetz numbers $L(f^k)$ and the Nielsen numbers $N(f^k)$ are different from the nature and not equal for maps $f$ on infra-solvmanifolds of type $\R$, we can utilize the arguments employed mainly in \cite{BaBo} and \cite[Chap.~III]{JM} for the Lefschetz numbers of iterations  
and thereby will study the asymptotic behavior of the sequence $\{N(f^k)\}$, the {essential periodic orbits} of $f$ and the homotopy minimal periods of $f$ by using the Nielsen theory of maps $f$ on infra-solvmanifolds of type $\R$. We give a linear lower bound for the number of essential periodic orbits of such a map (Theorem~\ref{BaBo4.2}), which sharpens well-known results of Shub and Sullivan for periodic points and of Babenko and Bogaty\v{\i} for periodic orbits. We also verify that a constant multiple of infinitely many prime numbers occur as homotopy minimal periods of such a map (Theorem~\ref{3.2.48}).

\section{Nielsen numbers $N(f^k)$}

Let $S$ be a connected and simply connected solvable Lie group. A discrete subgroup $\Gamma$ of $S$ is a {\it lattice} of $S$ if $\Gamma\bs{S}$ is compact, and in this case, we say that the quotient space $\Gamma\bs{S}$ is a {\it special} solvmanifold. Let $\Pi\subset\aff(S)$ be a torsion-free finite extension of the lattice $\Gamma=\Pi\cap S$ of $S$. That is, $\Pi$ fits the short exact sequence
$$
\CD
1@>>>S@>>>\aff(S)@>>>\aut(S)@>>>1\\
@.@AAA@AAA@AAA\\
1@>>>\Gamma@>>>\Pi@>>>\Pi/\Gamma@>>>1
\endCD
$$
Then $\Pi$ acts freely on $S$ and the manifold $\Pi\bs{S}$ is called an {\it infra-solvmanifold}. The finite group $\Phi=\Pi/\Gamma$ is the \emph{holonomy group} of $\Pi$ or $\Pi\bs{S}$. It sits naturally in $\aut(S)$. Thus every infra-solvmanifold $\Pi\bs{S}$ is finitely covered by the special solvmanifold $\Gamma\bs{S}$. An infra-solvmanifold $M=\Pi\bs{S}$ is of type $\R$ if $S$ is of \emph{type $\R$} or \emph{completely solvable}, i.e., if $\ad\, X:\frakS\to\frakS$ has only real eigenvalues for all $X$ in the Lie algebra $\frakS$ of $S$.

Recall that a connected solvable Lie group $S$ contains a sequence of closed subgroups
$$
1=N_1\subset\cdots\subset N_k=S
$$
such that $N_i$ is normal in $N_{i+1}$ and $N_{i+1}/N_i\cong\bbr$ or $N_{i+1}/N_i\cong S^1$. If the groups $N_1,\cdots,N_k$ are normal in $S$, the group $S$ is called \emph{supersolvable}. The supersolvable Lie groups are the Lie groups of type $\R$.

\begin{Lemma}[{\cite[Lemma~4.1]{Wilking}, \cite[Proposition~2.1]{JL15}}]\label{lemma:supersolvability}
For a connected Lie group $S$, the following are equivalent:
\begin{enumerate}
\item[$(1)$] $S$ is supersolvable.
\item[$(2)$] All elements of $\Ad(S)$ have only positive eigenvalues.
\item[$(3)$] $S$ is of type $\R$.
\end{enumerate}
\end{Lemma}

In this paper, we shall assume that $f:M\to M$ is a continuous map on an infra-solvmanifold $M=\Pi\bs{S}$ of type $\R$ with holonomy group $\Phi$. Then $f$ has an affine homotopy lift $(d,D):S\to S$, and so $f^k$ has an affine homotopy lift $(d,D)^k=(d',D^k)$ where $d'=dD(d)\cdots D^{k-1}(d)$.
By the averaging formulas for the Lefschetz and the Nielsen numbers $L(f^k)$ and $N(f^k)$ \cite[Theorem~4.2]{LL-Nagoya} and \cite{HL-Nagoya}, we have
\begin{align}
L(f^k)=\frac{1}{\#\Phi}\sum_{A\in\Phi}\det(I-A_*D_*^k),\quad
N(f^k)=\frac{1}{\#\Phi}\sum_{A\in\Phi}|\det(I-A_*D_*^k)|.\tag{AV}
\end{align}

Concerning the Nielsen numbers $N(f^k)$ of all iterates of $f$, we recall the following results.
These results about $N(f^k)$ are crucial in our discussion.

\begin{Thm}[{\cite[Theorem~11.4]{FL}}]\label{thm:Dold}
Let  $f:M\to M$ be a map on an infra-solvmanifold $M$ of type $\R$. Then
\begin{align}\label{DN}
\sum_{d\mid k}\mu\!\left(\frac{k}{d}\right)N(f^d)\equiv0\mod{k}\tag{DN}
\end{align}
for all $k>0$.
\end{Thm}

{Indeed, we have shown in \cite[Theorem~11.4]{FL} that the left-hand side is non-negative because it is equal to the number of isolated periodic points of $f$ with least period $k$.
By \cite[Lemma~2.1]{PW}, the sequence $\{N_k(f)\}$ is exactly realizable.}
\smallskip

Consider the sequences of algebraic multiplicities $\{A_k(f)\}$ and Dold multiplicities $\{I_k(f)\}$ associated to the sequence $\{N(f^k)\}$:
\begin{align}\label{eq:dold}
A_k(f)=\frac{1}{k}\sum_{d\mid k}\mu\!\left(\frac{k}{d}\right)N(f^d),\quad
I_k(f)=\sum_{d\mid k}\mu\!\left(\frac{k}{d}\right)N(f^d).
\end{align}
Then $I_k(f)=kA_k(f)$ and all $A_k(f)$ are integers by Theorem~\ref{thm:Dold}. From the M\"{o}bius inversion formula, we immediately have
\begin{align}\label{eq:algebraic period}
N(f^k)=\sum_{d\mid k}\ d\!\!~A_d(f).
\end{align}

\begin{Thm}[{\cite[Theorem~4.5]{DeDu}}]\label{thm:DD}
Let  $f:M\to M$ be a map on an infra-solvmanifold $M$ of type $\R$. Then the Nielsen zeta function of $f$
$$
N_f(z)=\exp\left(\sum_{k=1}^\infty\frac{N(f^k)}{k}z^k\right)
$$
is a rational function.
\end{Thm}

In fact, it is well-known that
$$
L_f(z)=\prod_{k=0}^m\det(I-z\cdot {f_*}_k)^{(-1)^{k+1}}
$$
where ${f_*}_k:H_k(M;\bbq)\to H_k(M;\bbq)$. Hence $L_f(z)$ is a rational function with coefficients in $\bbq$.
In \cite[Theorem~4.5]{DeDu}, it is shown that $N_f(z)$ is either
$$
N_f(z)=L_f((-1)^nz)^{(-1)^{p+n}}
$$
or
$$
N_f(z)=\left(\frac{L_{f_+}((-1)^nz)}{L_f((-1)^nz)}\right)^{(-1)^{p+n}}
$$
where $p$ is the number of real eigenvalues of $D_*$ which are $>1$ and $n$ is the number of real eigenvalues of $D_*$ which are $<-1$. Here $f_+$ is a lift of $f$ to a certain $2$-fold covering of $M$ which has the same affine homotopy lift $(d,D)$ as $f$. Consequently, $N_f(z)$ is a rational function with coefficients in $\bbq$.

On the other hand, since $N_f(0)=1$ by definition, $z=0$ is not a zero nor a pole of the rational function $N_f(z)$.
Thus we can write
\begin{align}\label{eq:rat}
N_f(z)=\frac{u(z)}{v(z)}=\frac{\prod_i(1-\beta_iz)}{\prod_j(1-\gamma_jz)}=\prod_{i=1}^r(1-\lambda_iz)^{-\rho_i}
\end{align}
with all $\lambda_i$ distinct nonzero algebraic integers{(see for example \cite{BL} or \cite[Theorem~2.1]{BaBo})}
 and $\rho_i$ nonzero integers. Taking log on both sides of the above identity, we obtain
$$
\sum_{k=1}^\infty \frac{N(f^k)}{k}z^k=\sum_{i=1}^r-\rho_i\log(1-\lambda_iz)
=\sum_{i=1}^r\rho_i\left(\sum_{k=1}^\infty\frac{(\lambda_iz)^k}{k}\right)
=\sum_{k=1}^\infty\frac{\sum_{i=1}^r\rho_i\lambda_i^k}{k}z^k.
$$
This induces
\begin{align}\label{N_k}
N(f^k)=\sum_{i=1}^{r(f)}\rho_i\lambda_i^k.\tag{N1}
\end{align}
Note that $r(f)$ is {the number of zeros and poles of $N_f(z)$}. Since $N_f(z)$ is a homotopy invariant, so is $r(f)$.
This argument tells us that whenever we have a rational expression of $N_f(z)$, we can write down all $N(f^k)$ directly from the expression. However even though we can compute all $N(f^k)$ using the averaging formula, it can be rather complicated to write down the rational expression of $N_f(z)$. 

Consider another generating function associated to the sequence $\{N(f^k)\}$:
$$
S_f(z)=\frac{d}{dz}\log N_f(z).
$$
Then it is easy to see that
\begin{align}\label{eq:gen fn}
S_f(z)=\sum_{k=1}^\infty N(f^k)z^{k-1}=\sum_{k=1}^\infty \sum_{i=1}^{r(f)}\rho_i\lambda_i^kz^{k-1}=\sum_{i=1}^{r(f)}\frac{\rho_i\lambda_i}{1-\lambda_iz},
\end{align}
which is a rational function with simple poles and integral residues, and $0$ at infinity.
The rational function $S_f(z)$ can be written as $S_f(z)=u(z)/v(z)$ where the polynomials $u(z)$ and $v(z)$ are of the form
$$
u(z)=N(f)+\sum_{i=1}^s a_iz^i,\quad v(z)=1+\sum_{j=1}^tb_jz^j
$$
with $a_i$ and $b_j$ integers, see $(3)\Rightarrow(5)$, Theorem~2.1 in \cite{BaBo} or \cite[Lemma~3.1.31]{JM}. Let $\tilde{v}(z)$ be the conjugate polynomial of $v(z)$, i.e., $\tilde{v}(z)=z^{t}v(1/z)$. Then the numbers $\{\lambda_i\}$ are the roots of $\tilde{v}(z)$, and $r(f)=t$.

The following can be found in the proof of $(3)\Rightarrow(5)$, Theorem~2.1 in \cite{BaBo}.
\begin{Lemma}\label{lemma:conjugate}
If $\lambda_i$ and $\lambda_j$ are roots of the {rational} polynomial $\tilde{v}(z)$ which are algebraically conjugate $($i.e., $\lambda_i$ and $\lambda_j$ are roots of the same irreducible polynomial$)$, then $\rho_i=\rho_j$.
\end{Lemma}

\begin{proof}
Let $\Sigma=\bbq(\lambda_1,\cdots,\lambda_r)\subset\bbc$ be the field of the rational polynomial $\tilde{v}(z)$ and let $\sigma$ be an automorphism of $\Sigma$ over $\bbq$, i.e., $\sigma$ is the identity on $\bbq$. The group of all such automorphisms is called the \emph{Galois group} of $\Sigma$. Since the $\sigma(\lambda_i)$ are again the roots of $\tilde{v}(z)$, we have $\sigma(\lambda_i)=\lambda_{\sigma(i)}$. That is, $\sigma$ induces a permutation $\sigma$ on $\{1,\cdots,r\}$. Applying $\sigma$ to the sequence $\{N(f^k)\}$, we obtain
\begin{align*}
\sigma\left(N(f^k)\right)&=\sigma\left(\sum_{i=1}^{r}\rho_i\lambda_i^k\right)
=\sum_{i=1}^{r}\rho_i\sigma(\lambda_i)^k=\sum_{i=1}^{r}\rho_i\lambda_{\sigma(i)}^k
=\sum_{i=1}^{r}\rho_{\sigma^{{-}1}(i)}\lambda_i^k.
\end{align*}
Since the $N(f^k)$ are integers, $\sigma(N(f^k))=N(f^k)$ and consequently
$$
\sum_{i=1}^{r}\rho_{i}\lambda_i^k=\sum_{i=1}^{r}\rho_{\sigma^{{-}1}(i)}\lambda_i^k.
$$
As a matrix form, we can write
$$
\left[\begin{matrix}
\lambda_1&\lambda_2&\cdots&\lambda_r\\
\lambda_1^2&\lambda_2^2&\cdots&\lambda_r^2\\
\vdots&\vdots&&\vdots\\
\lambda_1^r&\lambda_2^r&\cdots&\lambda_r^r\\
\end{matrix}\right]
\left[\begin{matrix}\rho_1\\ \rho_2\\ \vdots\\ \rho_r
\end{matrix}\right]
=\left[\begin{matrix}
\lambda_1&\lambda_2&\cdots&\lambda_r\\
\lambda_1^2&\lambda_2^2&\cdots&\lambda_r^2\\
\vdots&\vdots&&\vdots\\
\lambda_1^r&\lambda_2^r&\cdots&\lambda_r^r\\
\end{matrix}\right]
\left[\begin{matrix}\rho_{\sigma^{-1}(1)}\\ \rho_{\sigma^{-1}(2)}\\ \vdots\\ \rho_{\sigma^{-1}(r)}
\end{matrix}\right].
$$
Since the $\lambda_i$ are distinct, the matrices in this equation are nonsingular (the Vandermonde determinant). Thus $\rho_i=\rho_{\sigma^{-1}(i)}$ for all $i=1,\cdots,r$. On the other hand, it is known that the Galois group acts transitively on the set of algebraically conjugate roots. Since $\lambda_i$ and $\lambda_j$ are conjugate roots of $\tilde{v}(z)$, we can choose $\sigma$ in the Galois group so that $\sigma(\lambda_i)=\lambda_j$. Hence $\sigma(i)=j$ and so $\rho_i=\rho_j$.
\end{proof}

Let $\tilde{v}(z)=\prod_{\alpha=1}^s\tilde{v}_\alpha(z)$ be the decomposition of the monic integral polynomial $\tilde{v}(z)$ into irreducible polynomials $\tilde{v}_\alpha(z)$ of degree $r_\alpha$. Of course, $r=r(f)=\sum_{\alpha=1}^sr_\alpha$ and
\begin{align*}
\tilde{v}(z)&=z^r+b_1z^{r-1}+b_2z^{r-2}+\cdots+b_{r-1}z+b_r\\
&=\prod_{\alpha=1}^s(z^{r_\alpha}+b_1^\alpha z^{r_\alpha-1}+b_2^\alpha z^{r_\alpha-2}+\cdots+b_{r_\alpha-1}^\alpha z+b_{r_\alpha}^\alpha)
=\prod_{\alpha=1}^s\tilde{v}_\alpha(z).
\end{align*}
If $\{\lambda_i^{(\alpha)}\}$ are the roots of $\tilde{v}_\alpha(z)$, then the associated $\rho$'s are the same $\rho_\alpha$. Consequently, we can rewrite \eqref{N_k} as
\begin{align*}
N(f^k)&=\sum_{\alpha=1}^s\rho_\alpha\left(\sum_{i=1}^{r_\alpha}(\lambda_i^{(\alpha)})^k\right)\\
&=\sum_{\rho_\alpha>0}\rho_\alpha^+\left(\sum_{i=1}^{r_\alpha}(\lambda_i^{(\alpha)})^k\right)
-\sum_{\rho_\alpha<0}\rho_\alpha^-\left(\sum_{i=1}^{r_\alpha}(\lambda_i^{(\alpha)})^k\right).
\end{align*}
Consider the $r_\alpha\x r_\alpha$-{integral} square matrices
$$
M_\alpha=\left[\begin{matrix}
0&0&\cdots&0&-b^\alpha_{r_\alpha}\\
1&0&\cdots&0&\hspace{10pt}-b^\alpha_{r_\alpha-1}\\
\vdots&\vdots&&\vdots&\vdots\\
0&0&\cdots&0&-b^\alpha_2\\
0&0&\cdots&1&-b^\alpha_1
\end{matrix}\right].
$$
The characteristic polynomial is $\det(zI-M_\alpha)=\tilde{v}_\alpha(z)$ and therefore $\{\lambda_i^{(\alpha)}\}$ are the eigenvalues of $M_\alpha$. This implies that $N(f^k)=\sum_{\alpha=1}^s\rho_\alpha\ \tr M_\alpha^k$. Set
$$
M_+=\bigoplus_{\rho_\alpha>0} \rho_\alpha^+ M_\alpha,\qquad
M_-=\bigoplus_{\rho_\alpha<0} \rho_\alpha^- M_\alpha.
$$
Then
\begin{align}\label{N2}
N(f^k)=\tr M_+^k-\tr M_-^k=\tr(M_+\bigoplus -M_-)^k.\tag{N2}
\end{align}

From Theorem~\ref{thm:DD}, the rationality of the zeta functions on infra-solvmanifolds of type $\R$,
it was possible to derive the identities \eqref{N2} above.
Then we can reprove the Dold-Nielsen congruences \eqref{DN} of Theorem~\ref{thm:Dold}
by using the following fact ({\cite[Theorem~1]{Zarelua}}):
The Gauss congruence for the traces of powers of integer matrices
$$
\sum_{d\mid k}\mu\!\left(\frac{k}{d}\right)\tr\!(M^d)\equiv 0\mod k
$$
holds for all integer matrices $M$ and all natural $k$.
Recall also from \cite{Zarelua} that the following Euler congruence
$$
\tr\!(M^{p^r})\equiv \tr\!(M^{p^{r-1}})\mod p^r
$$
holds for all integer matrices $M$, all natural $r$, and all prime numbers $p$.
Furthermore, it is shown in \cite[Theorem~9]{Zarelua} that the above two assertions are equivalent.
Immediately we obtain the following Euler congruences for the Nielsen numbers,
which are equivalent to the Gauss congruences for the Nielsen numbers in Theorem~\ref{thm:Dold}.

\begin{Thm}\label{thm:EN}
Let  $f:M\to M$ be a map on an infra-solvmanifold $M$ of type $\R$. Then
\begin{align}\label{eq:EN}
N(f^{p^r})\equiv N(f^{p^{r-1}})\mod{p^r}\tag{EN}
\end{align}
for all $r>0$ and all prime numbers $p$.
\end{Thm}

We remark also that the rationality of the Nielsen zeta function $N_f(z)$
on an infra-solvmanifold of type $\R$
is equivalent to the existence of a self-map $g$ of some topological space $X$ such that
$N(f^k)=L(g^k)$ for all $k\ge1$.
In addition, due to the Gauss congruences \eqref{DN} in Theorem~\ref{thm:Dold}
we can choose $X$ to be a compact polyhedron, see \cite[Theorem~2.1 and Theorem(Dold)]{BaBo}.
Thus we can say that Nielsen theory on infra-solvmanifolds of type $\R$ is simpler than
we have expected and is reduced to Lefschetz theory but on different spaces that are not necessarily closed nor aspherical manifolds.

We will show in Proposition~\ref{epp} that if $A_k(f)\ne0$ then $N(f^k)\ne0$
and hence $f$ has an essential periodic point of period $k$. In the following we investigate some other necessary conditions under which $N(f^k)\ne0$. Recall that
\begin{align*}
N(f^k)&=\text{ the number of essential fixed point classes of $f^k$.}
\end{align*}
If $\bbf$ is a fixed point class of $f^k$, then $f^k(\bbf)=\bbf$ and the \emph{length} of $\bbf$ is the smallest number $p$ for which $f^p(\bbf)=\bbf$, written $p(\bbf)$. We denote by $\langle\bbf\rangle$ the $f$-orbit of $\bbf$, i.e., $\langle\bbf\rangle=\{\bbf,f(\bbf),\cdots, f^{p-1}(\bbf)\}$ where $p=p(\bbf)$. If $\bbf$ is essential, so is every $f^i(\bbf)$ and $\langle\bbf\rangle$ is an \emph{essential} periodic orbit of $f$ with length $p(\bbf)$ and $p(\bbf)\mid k$. These are variations of Corollaries~2.3, 2.4 and 2.5 of \cite{BaBo}.

\begin{Cor}\label{2.9}
If $r(f)\ne0$, then $N(f^i)\ne0$ for some $1\le i\le r(f)$. In particular, $f$ has at least $N(f^i)$ essential periodic points of period $i$ and an essential periodic orbit with the length $p\mid i, i\leq r(f)$.
\end{Cor}

\begin{proof}
Recall from \eqref{eq:gen fn} that
$$
S_f(z)=\sum_{i=1}^{r(f)}\frac{\rho_i\lambda_i}{1-\lambda_iz}=\sum_{k=1}^\infty N(f^k)z^{k-1}.
$$
Assume that $N(f)=\cdots =N(f^{r(f)})=0$. For simplicity, write the above identity as $S_f(z)=u(z)/v(z)=s(z)$ where $v(z)$ is a polynomial of degree $r(f)$ and $s(z)$ is the series of the right-hand side. Then $u(z)=v(z)s(z)$. A simple calculation shows that the higher order derivative of $v(z)s(z)$ up to order $r(f)-1$ at $0$ are all zero. Since $u(z)$ is a polynomial of degree $r(f)-1$, it shows that $u(z)=0$, a contradiction.
\end{proof}

Recalling the identity $N(f^k)=\sum_{i=1}^{r(f)}\rho_i\lambda_i^k$, we define
\begin{align}\label{eq:rho}
\rho(f)=\sum_{i=1}^{r(f)}\rho_i,\
M(f)=\max\left\{\sum_{\rho_i\ge0}\rho_i,-\sum_{\rho_j<0}\rho_j\right\},\
m(f)=\min\left\{\sum_{\rho_i\ge0}\rho_i,-\sum_{\rho_j<0}\rho_j\right\}.
\end{align}

\begin{Cor}
If $\rho(f)=0$ and $r(f)\ge1$, then $r(f)\ge2$ and $N(f^i)\ne0$ for some $1\le i< r(f)$. In particular, $f$ has at least $N(f^i)$ essential periodic points of period $i$ and an essential periodic orbit with the length $p\mid i, i\leq r(f)-1$.
\end{Cor}

\begin{proof}
The conditions $\rho(f)=0$ and $r(f)\ge1$ immediately implies that $r(f)\ge2$. Since $r(f)\ne0$ by the previous corollary there exists $i\in\{1,\cdots,r(f)\}$ such that $N(f^i)\ne0$.
Assume $N(f)=\cdots =N(f^{r(f)-1})=0$. So, $N(f^{r(f)})\ne0$. As in the proof of the above corollary, we consider $u(z)/v(z)=s(z)$ where $u(z)$ is a polynomial of degree $r(f)-1\ge1$ and $s(z)$ is a power series starting from the nonzero term $N(f^{r(f)})z^{r(f)-1}$. The derivative of order $r(f)-1$ on both sides of the identity $u(z)=v(z)s(z)$ yields that $N(f^{r(f)})=0$, a contradiction.
\end{proof}

\begin{Cor}
If $r(f)>0$, then $N(f^i)\ne0$ for some $1\le i\le M(f)$. In particular, $f$ has at least $N(f^i)$ essential periodic points of period $i$ and an essential periodic orbit with the length $p\mid i, i\leq M(f)$.
\end{Cor}

\begin{proof}
Assume that $N(f^k)=0$ for all $k=1,\cdots, M(f)$. From \eqref{N2}, we have $J_k:=\tr M_+^k=\tr M_-^k$. For simplicity, suppose
$$
m(f)=\sum_{\rho_i>0}\rho_i\le-\sum_{\rho_j<0}\rho_j=M(f).
$$
Then the matrix $M_+$ has size $m(f)$ and $M_-$ has size $M(f)$. We write the eigenvalues of $M_+$ and $M_-$ respectively as
$$
\mu_1,\cdots,\mu_{m(f)};\ \tilde\mu_1,\cdots,\tilde\mu_{M(f)}.
$$
Of course, $\{\mu_1,\cdots,\mu_{m(f)},\tilde\mu_1,\cdots,\tilde\mu_{M(f)}\}=\{\lambda_1,\cdots,\lambda_r\}$ as a set. Now the identities $\tr M_+^k=\tr M_-^k$ yield the $M(f)$ equations
$$
\mu_1^k+\cdots+\mu_{m(f)}^k+0+\cdots+0=\tilde\mu_1^k+\cdots+\tilde\mu_{M(f)}^k
$$
where $\mu_j=0$ when $m(f)<j\le M(f)$. By \cite[p.72, Corollary]{Bour}, there exists a permutation $\sigma$ on $\{1,\cdots,M(f)\}$ such that $\mu_i=\tilde\mu_{\sigma(i)}$. If $m(f)<M(f)$ then $0=\mu_{M(f)}=\tilde\mu_{\sigma(M(f))}=\lambda_j$ for some $j$, a contradiction. Hence $m(f)=M(f)$ and the $\lambda_i$'s associated to $\rho_i>0$ and the $\lambda_j$'s associated with $\rho_j<0$ are the same. This implies that the rational function $N_f(z)$ has the same poles and zeros of equal multiplicity and hence $N_f(z)\equiv1$, contradicting that $r(f)>0$.
\end{proof}

\section{Radius of convergence of $N_f(z)$}

From the Cauchy-Hadamard formula, we can see that the radii $R$ of convergence of the infinite series $N_f(z)$ and $S_f(z)$ are the same and given by
$$
\frac{1}{R}=\limsup_{k\to\infty}\left(\frac{N(f^k)}{k}\right)^{1/k}
=\limsup_{k\to\infty}N(f^k)^{1/k}.
$$

We will understand the radius $R$ of convergence from the identity $N(f^k)=\sum_{i=1}^{r(f)}\rho_i\lambda_i^k$. Recall that the $\lambda_i^{-1}$ are the poles or the zeros of the rational function $N_f(z)$.

\begin{Def}\label{def:lambda}
We define
$$
\lambda(f):=\max\{|\lambda_i|\mid i=1,\cdots,r(f)\}.
$$
If $r(f)=0$, i.e., if $N(f^k)=0$ for all $k>0$, then $N_f(z)\equiv1$ and $1/R=0$. In this case, we define customarily $\lambda(f)=0$.
When $r(f)\ne0$, we define
$$
n(f) := \#\{i\mid |\lambda_i|=\lambda(f)\}.
$$
\end{Def}

We shall assume now that $r(f)\ne0$ or $\lambda(f)>0$.
First we can observe easily the following:
\begin{enumerate}
\item $\limsup z_k^{1/k}=\limsup (r_ke^{i\theta_k})^{1/k}=\limsup r_k^{1/k}e^{i\theta_k/k}=\limsup r_k^{1/k}$.
\item $\limsup (\lambda^k)^{1/k}=|\lambda|$ by taking $z_k=\lambda^k$ in (1).
\item When $\lim z_k=0$ in (1), $\limsup z_k^{1/k}=0$.
\item $\lim(z_k+\rho)^{1/k}=\lim z_k^{1/k}$ when $\lim z_k=\infty$. For, in this case (1) induces
      $$\lim\left(\frac{z_k+\rho}{z_k}\right)^{1/k}=1.$$
\end{enumerate}
Assume $|\lambda_j|\ne\lambda(f)$ for some $j$; then we have
$$
\frac{N(f^k)}{\lambda_j^k}=\sum_{i\ne j}\rho_i\left(\frac{\lambda_i}{\lambda_j}\right)^k+\rho_j,\quad
\lim\sum_{i\ne j}\rho_i\left(\frac{\lambda_i}{\lambda_j}\right)^k=\infty.
$$
It follows from the above observations that $1/R=\limsup(\sum_{i\ne j}\rho_i\lambda_i^k)^{1/k}$. Consequently, we may assume that $N(f^k)=\sum_j\rho_j\lambda_j^k$ with all $|\lambda_j|=\lambda(f)$ and then we have
$$
\frac{1}{R}=\limsup \left(\sum_{|\lambda_j|=\lambda(f)} \rho_j\lambda_j^k\right)^{1/k}.
$$
Remark that if $\lambda(f)<1$ then $N(f^k)=\sum_{|\lambda_j|=\lambda(f)} \rho_j\lambda_j^k\to 0$ and so the sequence of integers are eventually zero, i.e., $N(f^k)=0$ for all $k$ sufficiently large. This shows that $1/R=0$ and furthermore, $N_f(z)$ is the exponential of a polynomial. Hence the rational function $N_f(z)$ has no poles and zeros. This forces $N_f(z)\equiv1$; hence $\lambda(f)=0$. If $\lambda(f)>1$, then $N(f^k)\to\infty$ and by L'H\^{o}pital's rule we obtain
\begin{align*}
\limsup_{k\to\infty}\frac{\log N(f^k)}{k}&=\limsup_{k\to\infty}\frac{\log\left(\sum_j\rho_j\lambda_j^k\right)}{k}
=\log\lambda(f)\ \Rightarrow\ \frac{1}{R}=\lambda(f).
\end{align*}
If $\lambda(f)=1$, then $N(f^k)\le \sum_j|\rho_j|<\infty$ is a bounded sequence and so it has a convergent subsequence. If $\limsup N(f^k)=0$, then $N(f^k)=0$ for all $k$ sufficiently large and so by the same reason as above, $\lambda(f)=0$, a contradiction. Hence $\limsup N(f^k)$ is a finite nonzero integer and so $1/R=1=\lambda(f)$.

Summing up, we have obtained that
\begin{Thm}\label{thm:Radius1}
Let $f$ be a map on an infra-solvmanifold of type $\R$. Let $R$ denote the radius of convergence of the Nielsen zeta function $N_f(z)$ of $f$. Then $\lambda(f)=0$ or $\lambda(f)\ge1$, and
\begin{align}\label{R1}
\frac{1}{R}=\lambda(f).\tag{R1}
\end{align}
In particular,
\begin{enumerate}
\item[$(1)$] $\lambda(f)=0$, which occurs if and only if $N_f(z)\equiv1$;
\item[$(2)$] if $\lambda(f)=1$, then $N(f^k)$ is a bounded sequence;
\item[$(3)$] $\lambda(f)>1$ if and only if $N(f^k)$ is an unbounded sequence.
\end{enumerate}
\end{Thm}

\begin{Rmk}
{In this paper, the number $\lambda(f)$ will play a similar role as the ``essential spectral radius" in \cite{JM} or the ``reduced spectral radius" in \cite{BaBo}.  Theorem~\ref{thm:Radius1} below shows that {$1/\lambda(f)$} is the ``radius" of the Nielsen zeta function $N_f(z)$. Note also that $\lambda(f)$ is a homotopy invariant.}
\end{Rmk}

By Theorem~\ref{thm:Radius1}, we see that the sequence $N(f^k)$ is either bounded or exponentially unbounded.

\begin{Rmk}
Recall from \eqref{eq:rat} and \eqref{eq:gen fn} that
\begin{align*}
&S_f(z)=\sum_{i=1}^{r(f)}\frac{\rho_i\lambda_i}{1-\lambda_iz},\\
&N_f(z)=\prod_{i=1}^{r(f)}(1-\lambda_iz)^{-\rho_i}=\frac{\prod_{\rho_j<0}(1-\lambda_jz)^{-\rho_j}}{\prod_{\rho_i>0}(1-\lambda_iz)^{\rho_i}}.
\end{align*}
These show that all of the $1/\lambda_i$ are the poles of $S_f(z)$, whereas the $1/\lambda_i$ with corresponding $\rho_i>0$ are the poles of $N_f(z)$. The radius of convergence of a power series centered at a point $a$ is equal to the distance from $a$ to the nearest point where the power series cannot be defined in a way that makes it holomorphic. Hence the radius of convergence of $S_f(z)$ is {$1/\lambda(f)$} and the radius of convergence of $N_f(z)$ is {$1/\max\{|\lambda_i|\mid \rho_i>0\}$}. In particular, we have shown that
$$
\lambda(f)=\max\{|\lambda_i|\mid i=1,\cdots,r(f)\}=\max\{|\lambda_i|\mid \rho_i>0\}.
$$
{Notice this identity in Example~\ref{Klein}.}
\end{Rmk}

On the other hand, we can understand the radius $R$ of convergence using the averaging formula. Compare our result with \cite[Theorem~7.10]{FL}.
Let $\{\mu_1,\cdots,\mu_m\}$ be the eigenvalues of $D_*$, counted with multiplicities, where $m$ is the dimension of the manifold $M$. We denote by $\sp(A)$ the \emph{spectral radius} of the matrix $A$ which is the largest modulus of an eigenvalue of $A$. From the definition, we have
\begin{align*}
&\sp(D_*)=\max\{|\mu_j|\mid i=1,\cdots,m\},\\
&\sp\left(\bigwedge D_*\right)=\begin{cases}
\sum_{|\mu|>1}|\mu|&\text{when $\sp(D_*)>1$;}\\
1&\text{when $\sp(D_*)\le1$.}
\end{cases}
\end{align*}
Note $|\det(I-D_*^k)|=\prod_{j=1}^m|1-\mu_j^k|$. If some $\mu_j=1$, then $\det(I-D_*^k)=0$ for all $k>0$ and hence $\limsup|\det(I-D_*^k)|^{1/k}=0$. Assume now all $\mu_j\ne1$; then $\det(I-D_*^k)\ne0$ for all $k>0$. Remark further from \cite[Theorem~4.1]{FL} that
\begin{align*}
\log\left(\limsup_{k\to\infty}|\det(I-D_*^k)|^{1/k}\right)
&=\limsup_{k\to\infty}\sum_{j=1}^m\frac{\log|1-\mu_j^k|}{k}\\
&=\begin{cases}
\sum_{|\mu|>1}\log|\mu|&\text{when $\sp(D_*)>1$}\\
0&\text{when $\sp(D_*)\le1$.}
\end{cases}
\end{align*}

Now we ascertain that if $D_*$ has no eigenvalue $1$, then
\begin{align}\label{R2}
\frac{1}{R}=\limsup_{k\to\infty}|\det(I-D_*^k)|^{1/k}=\sp\left(\bigwedge D_*\right).\tag{R2}
\end{align}
From the averaging formula, we have $N(f^k)\ge |\det(I-D_*^k)|/|\Phi|$. This induces
\begin{align*}
\frac{1}{R}=\limsup_{k\to\infty}N_k^{1/k}&\ge\limsup_{k\to\infty}\left(\frac{|\det(I-D_*^k)|}{|\Phi|}\right)^{1/k}\\
&=\limsup_{k\to\infty}|\det(I-D_*^k)|^{1/k}.
\end{align*}
Furthermore, for any $A\in\Phi$, we obtain (see the proof of \cite[Theorem~4.3]{FL})
$$
|\det(I-A_*D_*^k)|\le\prod_{j=1}^m(1+|\mu_j|^k)
$$
and hence from the averaging formula
$$
N(f^k)\le \prod_{j=1}^m(1+|\mu_j|^k)\ \Rightarrow
\frac{1}{R}\le\prod_{|\mu|>1}|\mu|.
$$
This finishes the proof of our assertion.

\bigskip

Following from \eqref{R1} and \eqref{R2}, we immediately have:
\begin{Thm}\label{thm:Radius2}
Let $f$ be a map on an infra-solvmanifold of type $\R$ with an affine homotopy lift $(d,D)$. Let $R$ denote the radius of convergence of the Nielsen zeta function of $f$. If $D_*$ has no eigenvalue $1$, then
$$
\frac{1}{R}=\sp\left(\bigwedge D_*\right)=\lambda(f).
$$
\end{Thm}

We recall that the asymptotic Nielsen number of $f$ is defined to be
$$
N^\infty(f):=\max\left\{1,\limsup_{k\to\infty}N(f^k)^{1/k}\right\}.
$$
We also recall that the most widely used measure for the complexity of a dynamical system is the topological entropy $h(f)$. A basic relation between these two numbers is $h(f)\ge \log N^\infty(f)$, which was found by Ivanov in \cite{I}. There is a conjectural inequality $h(f)\ge \log(\sp(f))$ raised by Shub \cite{Shub}. This conjecture was proven for all maps on infra-solvmanifolds of type $\R$, see \cite{mp,mp-a} and \cite{FL}. Now we can state about relations between $N^\infty(f)$, $\lambda(f)$ and $h(f)$.

\begin{Cor}
Let $f$ be a map on an infra-solvmanifold of type $\R$ with an affine homotopy lift $(d,D)$. If $D_*$ has no eigenvalue $1$, then
$$
N^\infty(f)=\lambda(f),\quad
h(f)\ge\log\lambda(f).
$$
\end{Cor}

\begin{proof}
From \cite[Theorem~4.3]{FL} and Theorem~\ref{thm:Radius2}, we have $N^\infty(f)=\sp(\bigwedge D_*)=\lambda(f)$. Hence by Ivanov's inequality, we obtain that $h(f)\ge \log N^\infty(f)=\log\lambda(f)$.
\end{proof}

The following example shows that the assumption in Theorem~\ref{thm:Radius2} and its Corollary that $1$ is not in the spectrum of $D_*$ is essential.
\begin{Example}\label{Klein}
Let $f:M\to M$ be a map of type $(r,\ell,q)$ on the Klein bottle $M$ induced by an affine map $(d,D):\bbr^2\to\bbr^2$. Recall from \cite[Theorem~2.3]{KLY} and its proof that $r$ is odd and $q=0$, and
\begin{align*}
(d,D)&=\begin{cases}
\left(-\frac{1}{2}\left[\begin{matrix}*\\\ell\end{matrix}\right],
\left[\begin{matrix}r&0\\0&q\end{matrix}\right]\right)&\text{when $r$ is odd;}\\
\left(\left[\begin{matrix}*\\ *\end{matrix}\right],
\left[\begin{matrix}r&0\\2\ell&0\end{matrix}\right]\right)&\text{when $r$ is even and $q=0$,}
\end{cases}\\
N(f^k)&=\begin{cases}
|q^k(1-r^k)|=\begin{cases}
q^k(r^k-1)&\text{when $r$ is odd and $qr>0$}\\
(-1)^kq^k(r^k-1)&\text{when $r$ is odd and $qr<0$}\\
\end{cases}\\
|1-r^k|=\begin{cases}1&\text{when $q=r=0$}\\
r^k-1&\text{when $r>0$ and $q=0$}\\
(-1)^k(r^k-1)&\text{when $r<0$ and $q=0$.}
\end{cases}
\end{cases}
\end{align*}
A simple calculation shows that
\begin{center}
\begin{tabular}{|c|c|c|c|c|c|}
\hline%
$q,r$ &$N(f^k)$&$N_f(z)$&$S_f(z)$&$(\lambda(f),\sp(\bigwedge D_*))$ \\%
\hline\hline%
$r=1$&$0$&$1$&$0$&$(0,\max\{1,|q|\})$\\%
\hline
$r$ \text{odd}, $qr>0$&$(qr)^k-q^k$&$\ds{\frac{1-qz}{1-qrz}}$&$\ds{\frac{qr}{1-qrz}-\frac{q}{1-qz}}$&$(qr,qr)$\\%
\hline
$r$ \text{odd}, $qr<0$&$(-qr)^k-(-q)^k$&$\ds{\frac{1+qz}{1+qrz}}$&$\ds{-\frac{qr}{1+qrz}+\frac{q}{1+qz}}$&$(-qr,-qr)$\\%
\hline
$q=0$, $r=0$&$1$&$\ds{\frac{1}{1-z}}$&$\ds{\frac{1}{1-z}}$&$(1,1)$\\%
\hline
$q=0$, $r>0$&$r^k-1$&$\ds{\frac{1+z}{1-rz}}$&$\ds{\frac{r}{1-rz}-\frac{1}{1+z}}$&$(r,r)$\\%
\hline
$q=0$, $r<0$&$(-r)^k-(-1)^k$&$\ds{\frac{1+z}{1+rz}}$&$\ds{-\frac{r}{1+rz}+\frac{1}{1+z}}$&$(-r,-r)$\\%
\hline
\end{tabular}
\end{center}

\smallskip

\noindent
These observations show that when one of the eigenvalues is $1$, the invariants $N_f(z)$, $\sp(\bigwedge D_*)$ and $\lambda(f)$ still strongly depend upon the other eigenvalue. Remark also in this example that the identity $\lambda(f)=\max\{|\lambda_i|\mid \rho_i>0\}$ holds.
\end{Example}

\section{Asymptotic behavior of the sequence $\{N(f^k)\}$}

In this section, we study the asymptotic behavior of the Nielsen numbers of iterates of maps on infra-solvmanifolds of type $\R$.

We can write $N(f^k)=\sum_{i=1}^{r(f)}\rho_i\lambda_i^k$ as $N(f^k)=\Gamma(f^k)+\Omega(f^k)$, where
\begin{align}\label{eq:lambda}
\Gamma_k=\Gamma(f^k)=\lambda(f)^k\left(\sum_{|\lambda_j|=\lambda(f)}\rho_je^{2i\pi(k\theta_j)}\right),\quad
\Omega_k=\Omega(f^k)=\sum_{|\lambda_i|<\lambda(f)}\rho_i\lambda_i^k. \tag{$\Lambda$}
\end{align}
Since $|\lambda_i|<\lambda(f)$, it follows that
$$
\frac{\Omega_k}{\lambda(f)^k}
=\sum_{|\lambda_i|<\lambda(f)}\rho_i\left(\frac{\lambda_i}{\lambda(f)}\right)^k
\to0.
$$

\begin{Thm}\label{thm:Asymptotic}
For a map $f$ of an infra-solvmanifold of type $\R$, one of the following three possibilities holds:
\begin{enumerate}
\item[$(1)$] $\lambda(f)=0$, which occurs if and only if $N_f(z)\equiv1$.
\item[$(2)$] The sequence $\{N(f^k)/\lambda(f)^k\}$ has the same limit points as a periodic sequence $\{\sum_j\alpha_j\epsilon_j^k\}$ where $\alpha_j\in\bbz,\epsilon_j\in\bbc$ and $\epsilon_j^q=1$ for some $q>0$.
\item[$(3)$] The set of limit points of the sequence $\{N(f^k)/\lambda(f)^k\}$ contains an interval.
\end{enumerate}
\end{Thm}

\begin{proof}
For simplicity, we denote $\lambda(f)$ by $\lambda_0$. Recall that $\lambda_0=0$ if and only if all $N(f^k)=0$ and otherwise, $\lambda_0\ge1$. Suppose that $\lambda_0\ge1$. We may assume that
$$
\lambda_1=\lambda_0e^{2i\pi\theta_1},\cdots, \lambda_{n(f)}=\lambda_0e^{2i\pi\theta_{n(f)}}
$$
are all the $\lambda_i$ of modulus $\lambda_0$ (see Definition~\ref{def:lambda} for $n(f)$).
From \eqref{eq:lambda}, we see that  the sequence $\{N(f^k)/\lambda_0^k\}$ has the same asymptotic behavior as the sequence
$$
\left\{\frac{\Gamma_k}{\lambda_0^k}\right\}=\left\{\sum_{j=1}^{n(f)}\rho_je^{2i\pi(k\theta_j)}\right\}.
$$

We consider the continuous function $T^{n(f)}\to[0,\infty)$ defined by
$(\xi_1,\cdots,\xi_{n(f)})\mapsto \big|\sum_{j=1}^{n(f)}\rho_je^{2i\pi\xi_j}\big|$.
For any subset $\calS$ of $\{1,\cdots,n(f)\}$, we have a sub-torus
$$
T^{|\calS|}=\{(\xi_1,\cdots,\xi_{n(f)})\in T^{n(f)}\mid \xi_j=0,\ \forall j\notin\calS\}.
$$
The restriction of the above continuous function to the sub-torus $T^{|\calS|}$
is continuous and has its maximum $m_\calS$ because $T^{|\calS|}$ is compact.

Now we show that either
\begin{enumerate}
\item[(i)] $\{\sum_{j=1}^{n(f)} \rho_je^{2i\pi k\theta_j}\}$ is periodic
\end{enumerate}
or
\begin{enumerate}
\item[(ii)] there exists $\calS\subset\{1,\cdots,n(f)\}$ such that
the sequence $\{\sum_{j=1}^{n(f)} \rho_je^{2i\pi k\theta_j}\}$ is dense in $[0,m_\calS]$.
\end{enumerate}
If $\dim_\bbz\{\theta_1,\cdots,\theta_{n(f)},1\}=1$,
then all $\theta_j$ are rational $p_j/q_j$.
Every $\lambda_j/\lambda_0=e^{2i\pi\theta_j}$ is a $q_j$th root of unity, and thus all $\lambda_j/\lambda_0=e^{2i\pi\theta_j}$ are roots of unity of degree $q=\lcm(q_1,\cdots,q_{n(f)})$, and hence the sequence $\{\sum_{j=1}^{n(f)}\rho_je^{2i\pi (k\theta_j)}\}$ is periodic of period $q$. This proves (2).

Suppose $\dim_\bbz\{\theta_1,\cdots,\theta_{n(f)},1\}=s>1$.
Then there exists the smallest subset $\calS=\{j_1,\cdots,j_s\}\subset\{1,\cdots,n(f)\}$
for which $\theta_{j_1},\cdots,\theta_{j_s},1$ are linearly independent over the integers.
This means that if $\sum_{i=1}^s\ell_i\theta_{j_i}=\ell$ with $\ell_i,\ell\in\bbz$
then $\ell_{j_1}=\cdots=\ell_{j_s}=\ell=0$.
Then it follows from \cite[Theorem~6, p.~\!91]{Ch} that
the sequence $(m\theta_{j_1},\cdots,m\theta_{j_s})$ is dense in $T^{|\calS|}$.
This proves (3) with $[0,m_\calS]$.
\end{proof}

\begin{Example}\label{exa:JM}
(1) Let $f$ be the identity on $S^1$.
Then $N(f^k)=0$ for all $k>0$ and so $\lambda(f)=0$ and $N_f(z)\equiv1$.

(2) Consider the map $f$ on $T^2$ induced by the matrix
$$
D=\left[\begin{matrix}\hspace{8pt}0&\hspace{8pt}1\\-1&-1\end{matrix}\right].
$$
The characteristic polynomial of $D$ is $t^2+t+1$.
The eigenvalues of $D$ are $\omega=-1/2+\sqrt{3}/2i$ and $\bar\omega$.
Hence
$$
L(f^k)=\det(I-D^k)=(1-\omega^k)(1-\bar\omega^k)=2-(\omega^k+\bar\omega^k)
=\begin{cases}
0&\text{when $k=3\ell$}\\
3&\text{when $k\ne3\ell$.}
\end{cases}
$$
So, $N(f^k)=L(f^k)$ for all $k>0$. Therefore, we have 
\begin{align*}
L_f(z)=N_f(z)&
=\frac{(1-\omega z)(1-\bar\omega z)}{(1-z)^2}.
\end{align*}
Consequently, we have that
\begin{align*}
&\lambda(f)=\max\{|\omega|,|\bar\omega|,1\}=1,\\
&\left\{\frac{\Gamma(f^k)}{\lambda(f)^k}\right\}=\left\{2-(\omega^k+\bar\omega^k)\right\}\
\text{ is $3$-periodic.}
\end{align*}

(3)
Let $f$ be the map on $T^4$ induced by the matrix
$$
D=\left[\begin{matrix}\hspace{8pt}0&1&0&0\\\hspace{8pt}0&0&1&0\\
\hspace{8pt}0&0&0&1\\-1&2&0&2\end{matrix}\right].
$$
The characteristic polynomial of $D$ is $t^4-2t^3-2t+1$.
This polynomial has two complex roots $\omega,\bar\omega$ of modulus $1$ and two real roots $\alpha,\beta$
with $0<\alpha<1$ and $\beta=1/\alpha$.
Note that this is an example showing that there are algebraic integers
$\omega,\bar\omega$
of modulus $1$ which are not roots of unity.
Indeed, $\omega=\frac{1-\sqrt{3}}{2}\pm yi\in S^1$ where $y^2=\frac{\sqrt{3}}{2}$.

Since
\begin{align*}
L(f^k)&=\det(I-D^k)=(1-\alpha^k)(1-\beta^k)(1-\omega^k)(1-\bar\omega^k),\\
N(f^k)&=|\det(I-D^k)|=-(1-\alpha^k)(1-\beta^k)(1-\omega^k)(1-\bar\omega^k)\\
&=-2+\omega^k+\bar\omega^k+2\alpha^k-(\alpha\omega)^k-(\alpha\bar\omega)^k\\
&\quad+2\beta^k-(\beta\omega)^k-(\beta\bar\omega)^k-2(\alpha\beta)^k+(\alpha\beta\omega)^k+(\alpha\beta\bar\omega)^k,
\end{align*}
we have
\begin{align*}
N_f(z)&=\frac{(1-z)^2(1-\alpha\omega z)(1-\alpha\bar\omega z)(1-\beta\omega z)(1-\beta\bar\omega z)(1-\alpha\beta z)^2}{(1-\omega z)(1-\bar\omega z)(1-\alpha z)^2(1-\beta z)^2(1-\alpha\beta\omega z)(1-\alpha\beta\bar\omega z)},\\
\lambda(f)&=\max\{1,\alpha,\beta,\alpha\beta\}=\beta,\quad
\frac{\Gamma(f^k)}{\lambda(f)^k}=2-\omega^k-\bar\omega^k.
\end{align*}
This is an example of Case (3) with $[0,2]$ in Theorem~\ref{thm:Asymptotic}.
\end{Example}

In Theorem~\ref{thm:Radius2}, we showed that if $D_*$ has no eigenvalue $1$ then $\lambda(f)=\sp(\bigwedge D_*)$. In Example~\ref{Klein}, we have seen that when $D_*$ has an eigenvalue $1$, there are maps $f$ for which $\lambda(f)\ne\sp(\bigwedge D_*)$ with $\lambda(f)=0$, and $\lambda(f)=\sp(\bigwedge D_*)$ with $\lambda(f)\ge1$. In fact, we prove in the following that the latter case is always true.

\begin{Lemma}\label{lemma:Radius3}
Let $f$ be a map on an infra-solvmanifold of type $\R$ with an affine homotopy lift $(d,D)$. If $\lambda(f)\ge1$, then $\lambda(f)=\sp(\bigwedge D_*)$.
\end{Lemma}

\begin{proof}
Since $\lambda(f)\ge1$, by Corollary~\ref{2.9}, $N(f^k)\ne0$ for some $k\ge1$ and then by the averaging formula, there is $B\in\Phi$ such that $\det(I-B_*D_*^k)\ne0$. Choose $\beta\in\Pi$ of the form $\beta=(b,B)$. Then $\beta(d,D)^k$ is another homotopy lift of $f^k$. We have observed above that there are numbers $\nu_1,\cdots,\nu_m$ such that $\det(I-B_*D_*^k)=\prod_{i=1}^m(1-\nu_i)$ and $\{(\mu_i^k)^\ell\}=\{\nu_i^\ell\}$ for some $\ell>0$. Since $\det(I-B_*D_*^k)\ne0$, $B_*D_*^k$ has no eigenvalue 1. Hence by Theorem~\ref{thm:Radius2}, we have $\lambda(f^k)=\sp(\bigwedge B_*D_*^k)$. Recall that $N(f^k)=\sum_{i=1}^{r(f)}\rho_i\lambda_i^k$ and $\lambda(f)=\max\{|\lambda_i|\}$. Since $\lambda(f)\ge1$, it follows that $\lambda(f^k)=\lambda(f)^k$. Observe further that $\sp(\bigwedge B_*D_*^k)=\prod_{|\mu_i|\ge1}|\nu_i|=\prod_{|\mu_i|\ge1}|\mu_i^k|=\sp(\bigwedge D_*)^k$. Consequently, we obtain the required identity $\lambda(f)=\sp(\bigwedge D_*)$.
\end{proof}

\begin{Example}
Consider the $3$-dimensional orientable flat manifold with fundamental group $\frakG_2$ generated by $\{t_1,t_2,t_3,\alpha\}$ where
\begin{align*}
&t_1=\left(\left[\begin{matrix}1\\0\\0\end{matrix}\right],I\right),\
t_2=\left(\left[\begin{matrix}0\\1\\0\end{matrix}\right],I\right),\
t_3=\left(\left[\begin{matrix}0\\0\\1\end{matrix}\right],I\right),\\
&\alpha=(a,A)=\left(\left[\begin{matrix}\frac{1}{2}\\0\\0\end{matrix}\right],
\left[\begin{matrix}1&\hspace{8pt}0&\hspace{8pt}0\\0&-1&\hspace{8pt}0\\0&\hspace{8pt}0&-1
\end{matrix}\right]\right).
\end{align*}
Thus,
$$
\frakG_2=\left\langle t_1,t_2,t_3,\alpha\mid [t_i,t_j]=1, \alpha^2=t_1, \alpha t_2\alpha^{-1}=t_2^{-1}, \alpha t_3\alpha^{-1}=t_3^{-1}\right\rangle.
$$
Let $\varphi:\frakG_2\to\frakG_2$ be any homomorphism. Every element of $\frakG_2$ is of the form $\alpha^kt_2^mt_3^n$. Thus $\varphi$ has the form
$$
\varphi(t_2)=\alpha^{k_2}t_2^{m_2}t_3^{n_2},\
\varphi(t_3)=\alpha^{k_3}t_2^{m_3}t_3^{n_3},\
\varphi(\alpha)=\alpha^{k}t_2^{m}t_3^{n}.
$$
The relations $\alpha t_2\alpha^{-1}=t_2^{-1}$ and $\alpha t_3\alpha^{-1}=t_3^{-1}$ yield that $k_2=k_3=0$, $(-1)^km_i=-m_i$ and $(-1)^kn_i=-n_i$. Hence when $k$ is even, we have $m_i=n_i=0$. Further, $\varphi(t_1)=t_1^k$.

Now we shall determine an affine map $(d,D)$ satisfying $\varphi(\beta)(d,D)=(d,D)\beta$
for all $\beta\in\frakG_2$.
\medskip

\noindent
{\sc Case $k=2\ell$.}\newline
In this case, we have $\varphi(t_2)=\varphi(t_3)=1$ and $\varphi(\alpha)=\alpha^kt_2^mt_3^n=t_1^\ell t_2^mt_3^n$, and hence we need to determine $(d,D)$ satisfying
\begin{align*}
&(d,D)=(d,D)(e_2,I)\Rightarrow D(e_2)={\bf0},\\
&(d,D)=(d,D)(e_3,I)\Rightarrow D(e_3)={\bf0},\\
&(\ell e_1+m e_2+n e_3,I)(d,D)=(d,D)(a,A)\\
&\qquad\Rightarrow \ell e_1+me_2+ne_3+d=d+D(a),\ D=DA.
\end{align*}
Hence the second and the third columns of $D$ must be ${\bf0}$ and so $D=DA$ is automatically satisfied and the first column of $D$ is $2\left[\begin{matrix}\ell&m&n\end{matrix}\right]^t$. That is,
$$
(d,D)=\left(\left[\begin{matrix}*\\{*}\\{*}\end{matrix}\right],
\left[\begin{matrix}2\ell&0&0\\2m&0&0\\2n&0&0\end{matrix}\right]\right).
$$
The eigenvalues of $D$ are $0$ (multiple) and $k=2\ell$, and $N(f^k)=|(2\ell)^k-1|$ and
$$
N_f(z)=\begin{cases}
\frac{1}{1-z}&\text{when $\ell=0$;}\\
\frac{1-z}{1-2\ell z}&\text{when $\ell\ge1$;}\\
\frac{1+z}{1+2\ell z}&\text{when $\ell\le-1$.}
\end{cases}
$$
It follows that $\lambda(f)=\max\{1,2|\ell|\}=\sp(\bigwedge D_*)$.
Moreover, the sequence $\{N(f^k)/\lambda(f)^k\}$ is asymptotically the constant sequence $\{1\}$.
In fact, if for example $\ell\ge1$ then $N(f^k)=(-1)\cdot1^k+1\cdot(2\ell)^k$, $\lambda(f)=2\ell$
and $\Gamma(f^k)=1\cdot(2\ell)^k$, hence $\Gamma(f^k)/\lambda(f)^k\equiv1$.
We then have Case (2) of Theorem~\ref{thm:Asymptotic}.
\medskip

\noindent
{\sc Case $k=2\ell+1$.}\newline
In this case, we have $\varphi(t_2)=t_2^{m_2}t_3^{n_2},\ \varphi(t_3)=t_2^{m_3}t_3^{n_3}$ and $\varphi(\alpha)=\alpha^kt_2^mt_3^n=\alpha t_1^\ell t_2^mt_3^n$, and hence we need to determine $(d,D)$ satisfying
\begin{align*}
&\varphi(t_2)(d,D)=(d,D)t_2\Rightarrow D(e_2)=m_2e_2+n_2e_3,\\
&\varphi(t_3)(d,D)=(d,D)t_3\Rightarrow D(e_3)=m_3e_2+n_3e_3,\\
&\varphi(\alpha)(d,D)=(d,D)\alpha\\
&\qquad\Rightarrow a+A(\ell e_1+me_2+ne_3)+A(d)=d+D(a),\ AD=DA.
\end{align*}
These yield
$$
(d,D)=\left(\left[\begin{matrix}*\\-\frac{m}{2}\\-\frac{n}{2}\end{matrix}\right],
\left[\begin{matrix}2\ell+1&0&0\\0&m_2&m_3\\0&n_2&n_3\end{matrix}\right]\right).
$$
Now we consider some explicit examples of such $D$. First we take $D$ to be
$$
D=\left[\begin{matrix}-1&\hspace{8pt}0&0\\\hspace{8pt}0&\hspace{8pt}d&e\\\hspace{8pt}0&-e&d\end{matrix}\right].
$$
Then $D$ has eigenvalues $-1$ and $\mu=d\pm ei$.
Clearly $N(f^k)=0$ for all even integers $k>0$.
For odd $k$, $D^k$ and $AD^k$ are respectively of the form
$$
D^k=\left[\begin{matrix}-1&\hspace{8pt}0&0\\\hspace{8pt}0&\hspace{8pt}x&y\\\hspace{8pt}0&-y&x\end{matrix}\right]\
\text{ and }\
AD^k=\left[\begin{matrix}-1&\hspace{8pt}0&\hspace{8pt}0\\\hspace{8pt}0&-x&-y\\\hspace{8pt}0&\hspace{8pt}y&-x\end{matrix}\right].
$$
Then
\begin{align*}
N(f^k)&=\frac{1}{2}\left\{|\det(I-D^k)|+|\det(I-AD^k)|\right\}\\
&=\frac{1}{2}\left\{2\left((1-x)^2+y^2\right)+2\left((1+x)^2+y^2\right)\right\}
=2(1+x^2+y^2).
\end{align*}
Here $x^2+y^2=\mu^k\bar\mu^k=|\mu|^{2k}$.
Consequently, $N(f^k)=2(1+|\mu|^{2k})$ for odd $k$. This yields that
\begin{align*}
N_f(z)&=\exp\left(\sum_{k=1}^\infty\frac{2(1+|\mu|^{2(2k-1)})}{2k-1}z^{2k-1}\right)\\
&=\exp\left(\sum_{k=1}^\infty\frac{2}{2k-1}z^{2k-1}+\sum_{k=1}^\infty\frac{2}{2k-1}(|\mu|^2z)^{2k-1}\right)\\
&=\exp\left(\log\frac{1+z}{1-z}+\log\frac{1+|\mu|^2z}{1-|\mu|^2z}\right)\\
&=\frac{(1+z)(1+|\mu|^2z)}{(1-z)(1-|\mu|^2z)}.
\end{align*}
Moreover, $\lambda(f)=\max\{1,|\mu|^2\}=\sp(\bigwedge D_*)$.
We can see also that the sequence
$\Gamma(f^k)/\lambda(f)^k$ is $(-1)^{k+1}+1$ if $|\mu|\ne1$ and $2((-1)^{k+1}+1)$ if $|\mu|=1$,
hence the sequence $\Gamma(f^k)/\lambda(f)^k$ is $2$-periodic.
We thus have Case (2) of Theorem~\ref{thm:Asymptotic}.

Secondly, we take $D$ to be
$$
D=\left[\begin{matrix}-1&0&0\\\hspace{8pt}0&d&e\\\hspace{8pt}0&e&f\end{matrix}\right].
$$
Let $D$ have eigenvalues $-1$ and $\mu_1$ and $\mu_2$.
For odd $k$, $D^k$ and $AD^k$ are respectively of the form
$$
D^k=\left[\begin{matrix}-1&0&0\\\hspace{8pt}0&x&y\\\hspace{8pt}0&y&z\end{matrix}\right]\
\text{ and }\
AD^k=\left[\begin{matrix}-1&\hspace{8pt}0&\hspace{8pt}0\\\hspace{8pt}0&-x&-y\\\hspace{8pt}0&-y&-z\end{matrix}\right].
$$
Then
\begin{align*}
N(f^k)&=\frac{1}{2}\left\{|\det(I-D^k)|+|\det(I-AD^k)|\right\}\\
&=\frac{1}{2}\left\{2\left((1-x)(1-z)-y^2\right)+2\left((1+x)(1+z)-y^2\right)\right\}
=2(1+xz-y^2).
\end{align*}
Here $xz-y^2=\mu_1^k\mu_2^k$. Hence $N(f^k)=(1+(-1)^{k+1})(1+\mu_1^k\mu_2^k)$.
This yields that
\begin{align*}
N_f(z)&=\exp\left(\sum_{k=1}^\infty\frac{2(1+(\mu_1\mu_2)^{2k-1})}{2k-1}z^{2k-1}\right)\\
&=\exp\left(\sum_{k=1}^\infty\frac{2}{2k-1}z^{2k-1}+\sum_{k=1}^\infty\frac{2}{2k-1}(\mu_1\mu_2 z)^{2k-1}\right)\\
&=\exp\left(\log\frac{1+z}{1-z}+\log\frac{1+\mu_1\mu_2 z}{1-\mu_1\mu_2 z}\right)\\
&=\frac{(1+z)(1+\mu_1\mu_2 z)}{(1-z)(1-\mu_1\mu_2 z)}.
\end{align*}
Observe also that $\lambda(f)=\max\{1, |\mu_1\mu_2|\}=\sp(\bigwedge D_*)$.
Whether $\mu_i$ are real or complex, we can see that the sequence  $N(f^k)/\lambda(f)^k$
is asymptotically periodic, and so we have Case (2) of Theorem~\ref{thm:Asymptotic}.
\end{Example}

It is important to know not only the rate of growth of the sequence $\{N(f^k)\}$ but also the frequency with which the largest Nielsen number is encountered. The following theorem shows that this sequence grows relatively dense. The following are variations of Theorem~2.7, Proposition~2.8 and Corollary~2.9 of \cite{BaBo}.

\begin{Thm}\label{thm:BaBo2.7}
Let $f:M\to M$ be a map on an infra-solvmanifold of type $\R$. If $\lambda(f)\ge1$, then there exist $\gamma>0$ and a natural number $N$ such that for any $m> N$ there is an $\ell\in\{0,1,\cdots,n(f)-1\}$ such that $N(f^{m+\ell})/\lambda(f)^{m+\ell}>\gamma$.
\end{Thm}

\begin{proof}
As in the proof of Theorem~\ref{thm:Asymptotic}, for any $k>0$, we can write $N(f^k)=\Gamma_k+\Omega_k$ so that
$\Gamma_k/\lambda(f)^k=\sum_{j=1}^{n(f)}\rho_je^{2i\pi(k\theta_j)}$.
Consider the following $n(f)$ consecutive identities
$$
\frac{\Gamma_{k+\ell}}{\lambda(f)^{k+\ell}}=\sum_{j=1}^{n(f)}\left(\rho_je^{2i\pi (k\theta_j)}\right)e^{2i\pi (\ell\theta_j)}, \quad\text{$\ell=0,\cdots,n(f)-1$.}
$$
Let $W=W(\theta_1,\cdots,\theta_{n(f)})$ be the Vandermonde operator on $\bbc^{n(f)}$
$$
W(\theta_1,\cdots,\theta_{n(f)})
=\left[\begin{matrix}
1&1&\cdots&1\\
e^{2i\pi\theta_1}&e^{2i\pi\theta_2}&\cdots&e^{2i\pi\theta_{n(f)}}\\
e^{2i\pi(2\theta_1)}&e^{2i\pi(2\theta_2)}&\cdots&e^{2i\pi(2\theta_{n(f)})}\\
\vdots&\vdots&&\vdots\\
e^{2i\pi(n(f)-1)\theta_1}&e^{2i\pi(n(f)-1)\theta_2}&\cdots&e^{2i\pi(n(f)-1)\theta_{n(f)}}
\end{matrix}\right]
$$
and let {$2\gamma=\bb W^{-1}\bb^{-1}$.} Then the vector $\vec\rho=\left(\rho_1e^{2i\pi(k\theta_1)},\cdots,\rho_{n(f)}e^{2i\pi(k\theta_{n(f)})}\right)$ satisfies $\bb W\vec\rho\bb\ge \bb W^{-1}\bb^{-1}\bb\vec\rho\bb=2\gamma\bb\vec\rho\bb\ge2\gamma\sqrt{n(f)}$. Thus there is at least one of the coordinates of the vector $W\vec\rho$ whose modulus is $\ge2\gamma$. That is, there is an $\ell\in\{0,1,\cdots,n(f)-1\}$ such that ${|\Gamma_{k+\ell}|}/{\lambda(f)^{k+\ell}}\ge2\gamma$.
On the other hand, since $\Omega_k/\lambda(f)^k\to 0$,
we can choose $N$ so large that $m>N\Rightarrow |\Omega_m|/\lambda(f)^m<\gamma$.

In all, whenever $m>N$ there is an $\ell\in\{0,1,\cdots,n(f)-1\}$ such that
\begin{align*}
\frac{N(f^{m+\ell})}{\lambda(f)^{m+\ell}}&\ge\frac{|\Gamma_{m+\ell}|}{\lambda(f)^{m+\ell}}
-\frac{|\Omega_{m+\ell}|}{\lambda(f)^{m+\ell}}
>2\gamma-\gamma=\gamma.
\end{align*}
This finishes the proof.
\end{proof}

\begin{Cor}\label{3.2.47}
If $\lambda(f)\ge1$, then we have
$$
\limsup_{k\to\infty}\frac{N(f^k)}{\lambda(f)^k}=\limsup_{k\to\infty}\frac{\Gamma(f^k)}{\lambda(f)^k}>0,\quad
\lim_{k\to\infty}\frac{\Omega(f^k)}{\lambda(f)^k}=0.
$$
\end{Cor}

\begin{Prop}
\label{BaBo2.8}
Let $f:M\to M$ be a map on an infra-solvmanifold of type $\R$ such that $\lambda(f)>1$. Then for any $\epsilon>0$, there exists $N$ such that if $N(f^m)/\lambda(f)^m\ge\epsilon$ for $m>N$, then the Dold multiplicity $I_m(f)$ satisfies
$$
|I_m(f)|\ge\frac{\epsilon}{2}\lambda(f)^m.
$$
\end{Prop}

\begin{proof}
From the definition of Dold multiplicity $I_k(f)$, we have
\begin{align*}
|I_k(f)|=\Big|\sum_{d\mid k}\mu\!\left(\frac{k}{d}\right)N(f^d)\Big|
\ge N(f^k)-\Big|\sum_{d\mid k, d\ne k}\mu\!\left(\frac{k}{d}\right)N(f^d)\Big|.
\end{align*}
Let $C$ be any number such that $2M(f)\le C$. Then for any $d>0$
$$
N(f^d)\le\sum_{i=1}^{r(f)}|\rho_i|\lambda(f)^d\le 2M(f)\lambda(f)^d\le C\lambda(f)^d.
$$
Thus we have
\begin{align*}
|I_k(f)|&\ge N(f^k)-C\sum_{d\mid k, d\ne k}\lambda(f)^d
\ge N(f^k)-C\tau(k)\lambda(f)^{k/2}\\
&=N(f^k)-C\frac{\tau(k)}{\lambda(f)^{k/2}}\lambda(f)^k
\end{align*}
where $\tau(k)$ is the number of divisors of $k$.
Since $\tau(k)\le2\sqrt{k}$, see \cite[Ex~3.2.17]{JM}, and since $\lambda(f)>1$, we have $\lim_{k\to\infty}\tau(k)/\lambda(f)^{k/2}=0$, and so there exists an integer $N$ such that $C\tau(k)/\lambda(f)^{k/2}<\epsilon/2$ for all $k>N$. Let $m>N$ such that $N(f^m)/\lambda(f)^m\ge\epsilon$. The above inequality induces the required inequality
\begin{align*}
|I_m(f)|&\ge \left(\frac{N(f^m)}{\lambda(f)^m}-C\frac{\tau(m)}{\lambda(f)^{m/2}}\right)\lambda(f)^m
\ge \frac{\epsilon}{2}\lambda(f)^m.\qedhere
\end{align*}
\end{proof}

Theorem~\ref{thm:BaBo2.7} and Proposition~\ref{BaBo2.8} imply immediately the following:
\begin{Cor}\label{BaBo2.9}
Let $f:M\to M$ be a map on an infra-solvmanifold of type $\R$ such that $\lambda(f)>1$.  Then there exist $\gamma>0$ and a natural number $N$ such that if $m\ge N$ then there exists $\ell$ with $0\le\ell\le n(f)-1$ such that $|I_{m+\ell}(f)|/\lambda(f)^{m+\ell}\ge \gamma/2$. In particular $I_{m+\ell}(f)\ne0$ and so $A_{m+\ell}(f)\ne0$.
\end{Cor}

\section{Essential periodic orbits}

In this section, we shall give an estimate from below the number of \emph{essential periodic orbits} of maps on infra-solvmanifolds of type $\R$.

First of all, we recall the following:
\begin{Thm}[{\cite{SS}}]
If $f:M\to M$ is a $C^1$-map on a smooth compact manifold $M$ and $\{L(f^k)\}$ is unbounded, then the set of periodic points of $f$, $\bigcup_k\fix(f^k)$, is infinite.
\end{Thm}

This theorem is not true for continuous maps. Consider the one-point compactification of the map of the complex plane $f(z)=2z^2/\bb z\bb$. This is a continuous degree two map of $S^2$ with only two periodic points. But $L(f^k)=2^{k+1}$.

However, when $M$ is an infra-solvmanifold of type $\R$, the theorem is true for all {continuous} maps $f$ on $M$. In fact, using the averaging formula (\cite[Theorem~4.3]{LL-Nagoya}, \cite{HL-Nagoya}), we obtain
\begin{align*}
|L(f^k)|\le\frac{1}{|\Phi|}\sum_{A\in\Phi}|\det(I-A_*D_*^k)|=N(f^k).
\end{align*}
If $L(f^k)$ is unbounded, then so is $N(f^k)$ and hence the number of essential fixed point classes of all $f^k$ is infinite. In fact, the inequality $|L(f)|\le N(f)$ for any map $f$ on an infra-solvmanifold was proved in \cite{Wong-crelle}.

\begin{Cor}
Let $f$ be a map on an infra-solvmanifold of type $\R$.
Suppose $\{N(f^k)\}$ is unbounded.
If every periodic point of $f$ is isolated, then the set of minimal periods of $f$ is infinite.
\end{Cor}
\begin{proof}
By assumption, each $\fix(f^m)$ consists of isolated points
and so $\fix(f^m)$ and hence $P_m(f)$ are finite.
If, in addition, $f$ has finitely many minimal periods then $f$ must have finitely many periodic points.
This implies that $\{N(f^k)\}$ is bounded, a contradiction.
\end{proof}

Recall that any map $f$ on an infra-solvmanifold of type $\R$ is homotopic to a map $\bar{f}$ induced by an affine map $(d,D)$. By \cite[Proposition~9.3]{FL}, every essential fixed point class of $\bar{f}$ consists of a single element $x$ with index {$\sgn\det(I-df_x)$}. Hence $N(f)=N(\bar{f})$ is the number of essential fixed point classes of $\bar{f}$. It is a classical fact that a homotopy between $f$ and $\bar{f}$ induces a one-one correspondence between the fixed point classes of $f$ and those of $\bar{f}$, which is index preserving. Consequently, we obtain
$$
|L(f^k)|\le N(f^k)\le \#\fix(f^k).
$$
This induces the following conjectural inequality (see \cite{Shub,SS}) for infra-solvmanifolds of type $\R$:
$$
\limsup_{k\to\infty}\frac{1}{k}\log|L(f^k)|\le \limsup_{k\to\infty}\frac{1}{k}\log\#\fix(f^k).
$$
\bigskip

We denote by $\calO(f,k)$ the set of all essential periodic orbits of $f$ with length $\le k$.
Thus
$$
\calO(f,k)=\{\langle\bbf\rangle\mid \bbf \text{ is an essential fixed point class of $f^m$ with $m\le k$}\}.
$$

\begin{Thm}\label{BaBo4.2}
Let $f$ be a map on an infra-solvmanifold of type $\R$. Suppose that the sequence $N(f^k)$ is unbounded. Then there exists a natural number $N_0$ such that
$$
k\ge N_0\Longrightarrow \#\calO(f,k)\ge \frac{k-N_0}{r(f)}.
$$
\end{Thm}

\begin{proof}
As mentioned earlier, we may assume that every essential fixed point class $\bbf$ of any $f^k$ consists of a single element $\bbf=\{x\}$. Denote by $\fix_e(f^k)$ the set of essential fixed point (class) of $f^k$. Thus $N(f^k)=\#\fix_e(f^k)$. Recalling also that $f$ acts on the set $\fix_e(f^k)$ from the proof of \cite[Theorem~11.4]{FL}, we have
$$
\calO(f,k)=\{\langle x\rangle\mid x \text{ is a essential periodic point of $f$ with length $\le k$}\}.
$$

Observe further that if $x$ is an essential periodic point of $f$ with minimal period $p$, then $x\in\fix_e(f^{q})$ if and only if $p\mid q$. The length of the orbit $\langle x\rangle$ of $x$ is $p$, and
\begin{align*}
&\fix_e(f^k)=\bigcup_{d\mid k}\ \fix_e(f^d),\\
&\fix_e(f^d)\bigcap\fix_e(f^{d'})=\fix_e(f^{\gcd(d,d')}).
\end{align*}
Recalling that
\begin{align*}
A_m(f)&=\frac{1}{m}\sum_{k\mid m}\mu\!\left(\frac{m}{k}\right)N(f^k)
=\frac{1}{m}\sum_{k\mid m}\mu\!\left(\frac{m}{k}\right)\#\fix_e(f^k),
\end{align*}
we define $A_m(f,\langle x\rangle)$ for any $x\in\bigcup_i\fix_e(f^i)$ to be
\begin{align*}
A_m(f,\langle x\rangle)&=\frac{1}{m}\sum_{k\mid m}\mu\!\left(\frac{m}{k}\right)\#\!\left(\langle x\rangle\cap\fix_e(f^k)\right).
\end{align*}
Then we have
\begin{align*}
A_m(f)=\sum_{\substack{\langle x\rangle\\ x\in\fix_e(f^m)}} A_m(f,\langle x \rangle).
\end{align*}

We begin with new notation. For a given integer $k>0$ and $x\in\bigcup_m\fix_e(f^m)$, let
\begin{align*}
&\calA(f,k)=\left\{m\le k\mid A_m(f)\ne0\right\},\\
&\calA(f,\langle x\rangle)=\left\{m\mid A_m(f,\langle x\rangle)\ne0\right\}.
\end{align*}
Remark that if $A_m(f)\ne0$ then there exists an essential periodic point $x$ of $f$ with period $m$ such that $A_m(f,\langle x\rangle)\ne0$. Consequently, we have
$$
\calA(f,k)\subset \bigcup_{\langle x\rangle\in\calO(f,k)} \calA(f,\langle x\rangle)
$$

Since $N(f^k)$ is unbounded, we have that $\lambda(f)>1$, see the observation just above Theorem~\ref{thm:Radius1}. {By Corollary~\ref{BaBo2.9}, there is $N_0$ such that if $n\ge N_0$ then there is $i$ with $n\le i\le n+n(f)-1$} such that $A_i(f)\ne0$. This leads to the estimate
$$
\#\calA(f,k)\ge\frac{k-N_0}{n(f)}\quad \forall k\ge N_0.
$$

Assume that $x$ has minimal period $p$. Then we have
\begin{align*}
A_m(f,\langle x\rangle)
&=\frac{1}{m}\sum_{p\mid n\mid m}\mu\!\left(\frac{m}{n}\right)\#\langle x\rangle
=\frac{p}{m}\sum_{p\mid n\mid m}\mu\!\left(\frac{m}{n}\right).
\end{align*}
Thus if $m$ is not a multiple of $p$ then by definition $A_m(f,\langle x\rangle)=0$. It is clear that $A_p(f,\langle x\rangle)=\mu(1)=1$, i.e., $p\in\calA(f,\langle x\rangle)$. Because $p\mid n\mid rp\Leftrightarrow n=r'p$ with $r'\mid r$, we have $A_{rp}(f,\langle x\rangle)=1/r \sum_{p\mid n\mid rp}\mu(rp/n)=1/r\sum_{r'\mid r}\mu(r/r')$ which is $0$ when and only when $r>1$. Consequently, $\calA(f,\langle x\rangle)=\{p\}$.

In all, we obtain the required inequality
\begin{align*}
\frac{k-N_0}{r(f)}&\le \#\calA(f,k)\le\#\calO(f,k).\qedhere
\end{align*}
\end{proof}

We consider the set of periodic points of $f$ with minimal period $k$
$$
P_k(f)=\fix(f^k)-\bigcup_{d\mid k, d<k}\fix(f^d).
$$
It is clear that $\fix(f)\subset \fix(f^2)$, i.e., any fixed point class of $f$ is naturally contained in a unique fixed point class of $f^2$. It is also known that $\fix_e(f)\subset \fix_e(f^2)$. We define
$$
\EP_k(f)=\fix_e(f^k)-\bigcup_{d\mid k, d<k}\fix_e(f^d),
$$
the set of \emph{essential periodic points} of $f$ with minimal period $k$. Because
$$
\fix_e(f^k)=\coprod_{d\mid k}\EP_d(f),
$$
we have
$$
N(f^k)=\#\fix_e(f^k)=\sum_{d\mid k}\# \EP_d(f).
$$

\begin{Prop}\label{epp}
For every $k>0$, we have
\begin{align*}
\# \EP_k(f)=\sum_{d\mid k}\mu\!\left(\frac{k}{d}\right)N(f^d)=I_k(f).
\end{align*}
In particular, if $I_k(f)\ne0$ then $N(f^k)\ne0$.
\end{Prop}

\begin{proof}
We apply the M\"{o}bius inversion formula to the above identity and then we obtain $\# \EP_k(f)=\sum_{d\mid k}\mu\!\left(\frac{k}{d}\right)N(f^d)$, which is exactly $I_k(f)$ by its definition.
\end{proof}

\begin{Def}
We consider the mod $2$ reduction of the Nielsen number $N(f^k)$ of $f^k$, written $N^{(2)}(f^k)$. A positive integer $k$ is a \emph{$N^{(2)}$-period} of $f$ if $N^{(2)}(f^{k+i})=N^{(2)}(f^i)$ for all $i\ge1$. We denote the minimal $N^{(2)}$-period of $f$ by $\alpha^{(2)}(f)$.
\end{Def}

\begin{Prop}[{\cite[Proposition~1]{Matsuoka}}]
Let $p$ be a prime number and let $A$ be a square matrix with entries in the field $\bbf_p$. Then there exists $k$ with $(p,k)=1$ such that
$$
\tr A^{k+i}=\tr A^i
$$
for all $i\ge1$.
\end{Prop}

Recalling \eqref{N2}: $N(f^k)=\tr M_+^k-\tr M_-^k=\tr(M_+\oplus -M_-)^k$, we can see easily that the minimal $N^{(2)}$-period $\alpha^{(2)}(f)$ always exists and must be an odd number.

Now we obtain a result which resembles \cite[Theorem~2]{Matsuoka}.
\begin{Thm}\label{3.3.11}
Let $f$ be a map on an infra-solvmanifold of type $\R$. Let $k>0$ be an odd number. Suppose that $\alpha^{(2)}(f)^2\mid k$ or $p\mid k$ where $p$ is a prime such that $p\equiv 2^i\mod{\alpha^{(2)}(f)}$ for some $i\ge0$. Then
$$
\#\{\langle x\rangle \mid x\in \EP_k(f)\}=\# \EP_k(f)/k
$$
is even.
\end{Thm}

\begin{proof}
By Proposition~\ref{epp}, $\# \EP_k(f)=I_k(f)$. Hence it is sufficient to show that $I_k(f)$ is even.

Let $\alpha=\alpha^{(2)}(f)$. Consider the case where $\alpha^2\mid k$. If $d\mid k$ and $\mu(k/d)\ne0$ then it follows that $\alpha\mid d$. By the definition of $\alpha$, $N^{(2)}(f^d)=N^{(2)}(f^\alpha)$. This induces that
\begin{align*}
I_k(f)&\equiv \sum_{d\mid k}\mu\!\left(\frac{k}{d}\right)N^{(2)}(f^d)
=N^{(2)}(f^\alpha)\sum_{d\mid k}\mu\!\left(\frac{k}{d}\right)=0 \mod2.
\end{align*}

Assume $p$ is a prime such that $p\mid k$ and $p\equiv2^i\!\!\mod{\alpha}$ for some $i\ge0$. Write $k=p^jr$ where $(p,r)=1$. Then
\begin{align*}
I_k(f)&
=\sum_{d\mid r}\mu\!\left(\frac{r}{d}\right)\left(\sum_{e\mid p^j}\mu\!\left(\frac{p^j}{e}\right)N((f^d)^e)\right)\\
&=\sum_{d\mid r}\mu\!\left(\frac{r}{d}\right)\left(\mu\!\left(1\right)N((f^d)^{p^j})+\mu\!\left(p\right)N((f^d)^{p^{j-1}})\right)\\
&=I_r(f^{p^j})-I_r(f^{p^{j-1}}).
\end{align*}
Since $\alpha$ is a $N^{(2)}$-period of $f$, it follows that the sequence $\{I_r(f^i)\}_i$ is $\alpha$-periodic in its $\mod2$ reduction, i.e., $I_r(f^{j+\alpha})\equiv I_r(f^i)\mod2$ for all $j\ge1$. Since $p\equiv2^i\mod{\alpha}$, we have  $I_r(f^{p^s})\equiv I_r(f^{2^{is}})\mod2$ for all $s\ge0$. Recall \cite[Proposition~5]{Browder}: For any square matrix $B$ with entries in the field $\bbf_p$ and for any $j\ge0$, we have $\tr B^{p^j}=\tr B$. Due to this result, we obtain
$$
N^{(2)}(f^{2^j})\equiv\tr(M_+\oplus -M_-)^{2^j}\equiv\tr(M_+\oplus -M_-)\equiv N^{(2)}(f) \mod2
$$
and it follows that $I_r(f^{2^{is}})\equiv I_r(f)\mod2$. Consequently, we have
$$
I_k(f)=I_r(f^{p^j})-I_r(f^{p^{j-1}})\equiv I_r(f^{2^{ij}})-I_r(f^{2^{i(j-1)}})\equiv I_r(f)-I_r(f)=0\mod2.
$$
This finishes the proof.
\end{proof}

\section{Homotopy minimal periods}

In this section, we study (homotopy) minimal periods of maps $f$ on infra-solvmanifolds of type $\R$. We want to determine $\HPer(f)$ only from the knowledge of the sequence $\{N(f^k)\}$. This approach was used in \cite{ABLSS, Hal, JL} for maps on tori, in \cite{JM1,JM2,JKM,JM,LZ-china,LZ-agt} for maps on nilmanifolds and some solvmanifolds, and in \cite{LL-JGP,LZ-jmsj} for expanding maps on infra-nilmanifolds.

\begin{Thm}\label{3.2.48}
Let $f$ be a map on an infra-solvmanifold of type $\R$ with $\lambda(f)>1$.
If the sequence $\{N(f^k)/\lambda(f)^k\}$ is asymptotically periodic,
then there exist an integer $m>0$ and an infinite sequence $\{p_i\}$ of primes
such that $\{mp_i\}\subset\Per(f)$. {Furthermore, $\{mp_i\}\subset\HPer(f)$.}
\end{Thm}

\begin{proof}
For simplicity, we denote $\tilde{N}(f^k)=\Gamma(f^k)/\lambda(f)^k$.
By Corollary~\ref{3.2.47}, $\lambda(f)\ge1$ implies that $\limsup N(f^k)/\lambda(f)^k=\limsup \tilde{N}(f^k)>0$.
We consider the condition that the sequence $\{\tilde{N}(f^k)\}$ is periodic.
By Theorem~\ref{thm:Asymptotic},
we can choose $q$ such that
the sequence $\{\tilde{N}(f^k)\}$ is $q$-periodic and nonzero.
Consequently, there exists $m$ with $1\le m\le q$ such that $\tilde{N}(f^m)\ne0$.

Let $h=f^m$. Then $\lambda(h)=\lambda(f^m)=\lambda(f)^m\ge1$.
The periodicity $\tilde{N}(f^{m+\ell q})=\tilde{N}(f^m)$ induces
$\tilde{N}(h^{1+\ell q})=\tilde{N}(h)$ for all $\ell>0$.
By Corollary~\ref{3.2.47} again,
we can see that there exists $\gamma>0$ such that $N(h^{1+\ell q})\ge \gamma\lambda(h)^{1+\ell q}>0$ for all $\ell$ sufficiently large.
Since $\lambda(f)>1$, we have $\lambda(h)>1$ and
it follows from Proposition~\ref{BaBo2.8} that the Dold multiplicity $I_{1+\ell q}(h)$ satisfies $|I_{1+\ell q}(h)|\ge (\gamma/2)\lambda(h)^{1+\ell q}$ when $\ell$ is sufficiently large.

According to Dirichlet prime number theorem, since $(1,q)=1$, there are infinitely many primes $p$ of the form $1+\ell q$. Consider all primes $p_i$ satisfying $|I_{p_i}(h)|\ge (\gamma/2)\lambda(h)^{p_i}$. Remark that for any prime number $p$,
\begin{align*}
I_p(h)&=\sum_{d\mid p}\mu\!\left(\frac{p}{d}\right)N(h^d)=\mu(p)N(h)+\mu(1)N(h^p)=N(h^p)-N(h)\\
&=\#\fix_e(h^d)-\#\fix_e(h)=\#\!\left(\fix_e(h^p)-\fix_e(h)\right)
\end{align*}
where the last identity follows from that fact that $\fix_e(h)\subset\fix_e(h^p)$. Since $p$ is a prime, the set $\fix_e(h^p)-\fix_e(h)$ consists of essential periodic points of $h$ with minimal period $p$.

Because $|I_{p_i}(h)|>0$, each $p_i$ is the minimal period of some essential periodic point of $h$. Thus $mp_i$ is a period of $f$. This means that $m_ip_i$ is the minimal period of $f$ for some $m_i$ with $m_i\mid m$. Choose a subsequence $\{m_{i_k}\}$ of the sequence $\{m_i\}$ bounded by $m$ which is constant, say $m_0$. Consequently, the infinite sequence $\{m_0p_{i_k}\}$ consists of minimal periods of $f$, or $\{m_0p_i\}\subset\Per(f)$.

{These arguments also work for all maps homotopic to $f$. Hence $\{m_0p_i\}\subset\HPer(f)$,}
which completes the proof.

We next consider the condition that the sequence $\{\tilde{N}(f^k)\}$ contains an interval.
This means by Theorem~\ref{thm:Asymptotic} that the set of limit points
of the sequence $\{N(f^k)/\lambda(f)^k\}$ contains an interval.
\end{proof}

The following example shows that the condition $\lambda(f)>1$ in Theorem~\ref{3.2.48} is essential.
This condition is equivalent to the unboundedness of the sequence $\{N(f^k)\}$ by Theorem~\ref{thm:Radius1}.

\begin{Example}
Consider the map $f$ on $T^2$ induced by the matrix
$$
D=\left[\begin{matrix}\hspace{8pt}0&\hspace{8pt}1\\-1&-1\end{matrix}\right].
$$
We have observed in Example~\ref{exa:JM} that
\begin{align*}
&\{N(f^k)\}=\{2-(\omega^k+\bar\omega^k)\}\ \text{ is bounded where $\omega=-{1}/{2}+{\sqrt{3}i}/{2}$},\\
&\lambda(f)=\max\{|\omega|,|\bar\omega|,1\}=1,\\
&\{\tilde{N}(f^k)\}=\{2-(\omega^k+\bar\omega^k)\}\ \text{ is $3$-periodic.}
\end{align*}
Observe also that since $f^3=\id$ we have $\Per(f)\subset\{1,2,3\}$.
In fact, we can see that $\Per(f)=\{1,3\}$.
\end{Example}

{In the proof of Theorem~\ref{3.2.48}, we have shown the following, which proves that the algebraic period is a homotopy minimal period when it is a prime number.}

\begin{Cor}\label{3.2.50}
Let $f$ be a map on an infra-solvmanifold of type $\R$.
For any prime $p$, if $A_p(f)\ne0$ then {$p\in\HPer(f)$}.
\end{Cor}

\begin{Cor}\label{3.2.51}
Let $f$ be a map on an infra-solvmanifold of type $\R$ with $\lambda(f)>1$.
If the sequence $\{N(f^k)\}$ is eventually monotone increasing,
then there exists $N$ such that the set {$\HPer(f)$} contains all primes larger than $N$.
\end{Cor}

\begin{proof}
Since $\lambda(f)>1$, by Theorem~\ref{thm:BaBo2.7} there exist $\gamma>0$ and $N$ such that if $k>N$ then there exists $\ell=\ell(k)<r(f)$ such that $N(f^{k-\ell})/\lambda(f)^{k-\ell}>\gamma$. Then for all sufficiently large $k$, the monotonicity induces
\begin{align*}
\frac{N(f^k)}{\lambda(f)^k}\ge \frac{N(f^{k-\ell})}{\lambda(f)^k}=\frac{N(f^{k-\ell})}{\lambda(f)^{k-\ell}\lambda(f)^\ell}
\ge\frac{\gamma}{\lambda(f)^\ell}\ge\frac{\gamma}{\lambda(f)^{r(f)}}.
\end{align*}
Applying Proposition~\ref{BaBo2.8} with $\epsilon=\gamma/\lambda(f)^{r(f)}$, we see that $I_k(f)\ne0$ and so $A_k(f)\ne0$ for all $k$ sufficiently large. Now our assertion follows from Corollary~\ref{3.2.50}.
\end{proof}

{We next recall the following:}
\begin{Def}
A map $f:M\to M$ is \emph{essentially reducible} if any fixed point class of $f^k$ being contained in an essential fixed point class of $f^{kn}$ is essential, for any positive integers $k$ and $n$. The space $M$ is \emph{essentially reducible} if every map on $M$ is essentially reducible.
\end{Def}

\begin{Lemma}[{\cite[Proposition~2.2]{ABLSS}}]\label{Hal}
Let $f:M\to M$ be an essentially reducible map. If
$$
\sum_{\frac{m}{k}:\text{ {\rm prime}}} N(f^k)< N(f^m),
$$
then any map which is homotopic to $f$ has a periodic point with minimal period $m$, i.e., $m\in\HPer(f)$.
\end{Lemma}

\begin{Lemma}\label{er}
Every infra-solvmanifold of type $\R$ is essentially reducible.
\end{Lemma}

\begin{proof}
Let $f:M\to M$ be a map on an infra-solvmanifold $M=\Pi\bs{S}$ of type $\R$. Then $\Pi$ fits a short exact sequence
$$
1\lra\Gamma\lra\Pi\lra\Phi\lra1
$$
where $\Gamma=\Pi\cap S$ and the holonomy group $\Phi$ of $\Pi$ naturally sits in $\aut(S)$. By \cite[Lemma~2.1]{LL-Nagoya}, we know that $\Pi$ has a fully invariant subgroup $\Lambda$ of finite index and $\Lambda\subset\Gamma$. Therefore $\Lambda\subset\Gamma\subset S$ and $\bar{M}=\Lambda\bs{S}$ is a special solvmanifold which covers $M$. Since $\Lambda$ is a fully invariant subgroup of $\Pi$, it follows that any map $f:M\to M$ has a lifting $\bar{f}:\bar{M}\to\bar{M}$, and $\bar{M}$ is a regular covering of $M$. By \cite[Corollary~4.5]{HK}, $\bar{f}$ is essentially reducible and then by \cite[Proposition~2.4]{LZ-jmsj}, $f$ is essentially reducible.
\end{proof}

{We can not only extend but also strengthen Corollary~\ref{3.2.51} as follows:}
\begin{Prop}\label{prime power}
Let $f$ be a map on an infra-solvmanifold of type $\R$. Suppose that the sequence $\{N(f^k)\}$ is strictly monotone increasing. Then:
\begin{enumerate}
\item[$(1)$] All primes belong to $\HPer(f)$.
\item[$(2)$] There exists $N$ such that if $p$ is a prime $>N$ then $\{p^n\mid n\in\bbn\}\subset\HPer(f)$.
\end{enumerate}
\end{Prop}

\begin{proof}
Observe that for any prime $p$
$$
N(f^p)-\sum_{\frac{p}{k}:\text{ {\rm prime}}} N(f^k)=N(f^p)-N(f)=I_p(f).
$$
The strict monotonicity implies $A_p(f)=pI_p(f)>0$, and hence $p\in\HPer(f)$ by Corollary~\ref{3.2.50}.
This proves (1).

The strict monotonicity of $\{N(f^k)\}$ implies that $\lambda(f)>1$.
Under this assumption, we have shown in the proof of Corollary~\ref{3.2.51} that there exists $N$ such that $k>N\Rightarrow I_k(f)>0$. Let $p$ be a prime $>N$ and $n\in\bbn$. Then
\begin{align*}
&N(f^{p^n})-\sum_{\frac{p^n}{k}:\text{ {\rm prime}}} N(f^k)
=\sum_{i=0}^n I_{p^{i}}(f)-N(f^{p^{n-1}})=I_{p^n}(f)>0.
\end{align*}
By Lemma~\ref{Hal}, we have $p^n\in\HPer(f)$, which proves (2).
\end{proof}

For the set of algebraic periods $A(f)=\{m\in\bbn\mid A_m(f)\ne0\}$,
its lower density $\DA(f)$ was introduced in \cite[Remark~3.1.60]{JM}:
$$
\DA(f)=\liminf_{k\to\infty}\frac{\#(A(f)\cap[1,k])}{k}.
$$
We can consider as well the lower densities of $\Per(f)$ and $\HPer(f)$, see also \cite{Zhao et al}:
\begin{align*}
&\DP(f)=\liminf_{k\to\infty}\frac{\#(\Per(f)\cap[1,k])}{k},\\
&\DH(f)=\liminf_{k\to\infty}\frac{\#(\HPer(f)\cap[1,k])}{k}.
\end{align*}
Since $I_k(f)=\#\EP_k(f)$ by Proposition~\ref{epp}, it follows that $\calA(f)\subset\HPer(f)\subset\Per(f)$. Hence we have $\DA(f)\le\DH(f)\le\DP(f)$.

By Corollary~\ref{BaBo2.9}, when $\lambda(f)>1$, we have a natural number $N$ such that
if $m\ge N$ then there is $\ell$ with $0\le\ell<n(f)$ such that $A_{m+\ell}(f)\ne0$.
This shows that $\DA(f)\ge 1/n(f)$.

On the other hand, by Theorem~\ref{thm:Asymptotic}, we can obtain the following:
If $\lambda(f)=0$ then $N(f^k)=0$ and $A_k(f)=0$ for all $k$, which shows that $\DA(f)=0$.
Consider first Case (2), i.e., the sequence $\{N(f^k)/\lambda(f)^k\}$ is asymptotically
a periodic and nonzero sequence $\{\sum_{j=1}^{n(f)}\rho_je^{2i\pi(k\theta_j)}\}$ of some period $q$.
Now from the identity \eqref{eq:algebraic period} it follows that $\DA(f)\ge1/q$.
Finally consider Case (3). Then the sequence $\{N(f^k)/\lambda(f)^k\}$ asymptotically has
a subsequence $\{\sum_{j\in\calS}\rho_je^{2i\pi(k\theta_j)}\}$ where $\calS=\{j_1,\cdots,j_s\}$
and $\{\theta_{j_1},\cdots,\theta_{j_s},1\}$ is linearly independent over the integers.
Therefore by \cite[Theorem~6, p.~\!91]{Ch}, the sequence $(k\theta_{j_1},\cdots,k\theta_{j_s})$
is uniformly distributed in $T^{|\calS|}$.
From the identity \eqref{eq:algebraic period} it follows that $\DA(f)=1$
(see \cite[Remark~3.1.60]{JM}).

Theorem~\ref{3.2.48} studies the homotopy minimal periods for maps of Case (2) in Theorem~\ref{thm:Asymptotic}.
Now we can state immediately the following result for maps of Case (3) in Theorem~\ref{thm:Asymptotic}.

\begin{Cor}\label{cor:Case 3}
Let $f$ be a map on an infra-solvmanifold of type $\R$.
Suppose that the sequence $\{N(f^k)/\lambda(f)^k\}$
asymptotically contains an interval. Then $\DA(f)=\DH(f)=\DP(f)=1$.
\end{Cor}

\begin{Cor}\label{cofinite}
Let $f$ be a map on an infra-solvmanifold of type $\R$. Suppose that the sequence $\{N(f^k)\}$ is unbounded and eventually  monotone increasing. Then $\HPer(f)$ is cofinite and $\DA(f)=\DH(f)=\DP(f)=1$.
\end{Cor}

\begin{proof}
Under the same assumption, we have shown in the proof of Corollary~\ref{3.2.51} that there exists $N$ such that if $k>N$ then $I_k(f)>0$. This means $\EP_k(f)$ is nonempty by Proposition~\ref{epp} and hence $k\in\HPer(f)$.
\end{proof}

Now we can prove the main result of \cite{LZ-jmsj}.
\begin{Cor}[{\cite[Theorem~4.6]{LL-JGP}, \cite[Theorem~3.2]{LZ-jmsj}}]
Let $f$ be an expanding map on an infra-nilmanifold. Then $\HPer(f)$ is cofinite.
\end{Cor}

\begin{proof}
Since $f$ is expanding, we have that $\lambda(f)=\sp(\bigwedge D_*)>1$.
For any $k>0$, we can write as before $N(f^k)=\Gamma(f^k)+\Omega(f^k)$
so that $\Omega(f^k)\to0$ and $\Gamma(f^k)\to\infty$ as $k\to\infty$.
This implies that $N(f^k)$ is eventually monotone increasing.
The assertion follows from Corollary~\ref{cofinite}.
\end{proof}

\end{document}